\documentclass[12pt,sort&compress]{elsarticle}

\usepackage{bm}

\usepackage{mathrsfs}
\usepackage{amsmath}
\usepackage{amsfonts}
\usepackage{amsthm}
\usepackage{color}
\usepackage{xcolor}
\usepackage{caption}
\usepackage{subcaption}
\usepackage{float}
\usepackage{url}
\usepackage{multicol}
\usepackage[top=1in,bottom=1in,left=1in,right=1in]{geometry}
\usepackage[pdfborder={0 0 0},colorlinks,allcolors=blue]{hyperref}
\usepackage{txfonts}
\usepackage{setspace}
\usepackage{arydshln}
\usepackage{enumitem}
\usepackage{lipsum}% for dummy text
\usepackage{siunitx,booktabs}
\usepackage{nth}
\usepackage{accents} % Accents: circle

\theoremstyle{definition}%plain/definition/remark

\newtheorem{theorem}{Theorem}%[section]
\newtheorem{lemma}{Lemma}%[section]
\newtheorem{assumption}{Assumption}%[section]
\theoremstyle{definition}%plain/definition/remark
\newtheorem{remark}{Remark}%[section]

\allowdisplaybreaks

\newcommand{\vertiii}[1]{{\left\vert\kern-0.2ex\left\vert\kern-0.2ex\left\vert #1 
    \right\vert\kern-0.2ex\right\vert\kern-0.2ex\right\vert}}

\usepackage{listings}
\definecolor{light-gray}{gray}{0.95}
\begin{document}
\begin{frontmatter}

\title{Variational multiscale modeling with discretely divergence-free subscales}

\author[colorado]{John~A.~Evans\corref{cor1}}
\ead{john.a.evans@colorado.edu}
\cortext[cor1]{Corresponding author}
\author[ucsd]{David~Kamensky}
\author[brown]{Yuri~Bazilevs}

\address[colorado]{Department of Aerospace Engineering Sciences, University of Colorado at Boulder, 429 UCB, Boulder, CO 80309, USA}
\address[ucsd]{Department of Mechanical and Aerospace Engineering, University of California, San Diego, 9500 Gilman Drive, Mail Code 0411, La Jolla, CA 92093, USA}
\address[brown]{School of Engineering, Brown University, 184 Hope St., Providence, RI 02912, USA}

\journal{Computers \& Mathematics with Applications}

\begin{abstract}
We introduce a residual-based stabilized formulation for incompressible Navier--Stokes flow that maintains discrete (and, for divergence-conforming methods, strong) mass conservation for inf-sup stable spaces with $H^1$-conforming pressure approximation, while providing optimal convergence in the diffusive regime, robustness in the advective regime, and energetic stability.  The method is formally derived using the variational multiscale (VMS) concept, but with a discrete fine-scale pressure field which is solved for alongside the coarse-scale unknowns such that the coarse and fine scale velocities separately satisfy discrete mass conservation.  We show energetic stability for the full Navier--Stokes problem, and we prove convergence and robustness for a linearized model (Oseen flow), under the assumption of a divergence-conforming discretization.  Numerical results indicate that all properties extend to the fully nonlinear case and that the proposed formulation can serve to model unresolved turbulence.
\end{abstract}


\end{frontmatter}

\section{Introduction}\label{sec:introduction}
Over the past two decades, a substantial amount of evidence has been published to suggest that residual-based stabilized formulations for incompressible flow are effective as implicit large-eddy simulation (LES) models, both in fundamental turbulence studies \cite{Hoffman2006,Bazilevs07b,Bazilevs09d} and industrial-scale applications \cite{Bazilevs10a,Bazilevs10b,Hoffman2016}.  The connection between residual-based stabilization and LES was illuminated by variational multiscale (VMS) theory \cite{Hug98,HugScoFr04,masud2006multiscale,Hughes07b}, in which unresolved fine-scale solution components are driven by the residual from introducing the resolved coarse-scale solution into the partial differential equation (PDE) system.

A shortcoming of residual-based stabilized formulations is that they typically disrupt the discrete continuity equation: the velocity divergence integrated against an arbitrary member of some pressure test space is no longer guaranteed to be zero, as it would be in Galerkin's method.  This is true, for instance, of the popular streamline upwind Petrov--Galerkin (SUPG)/pressure-stabilized Petrov--Galerkin (PSPG) \cite{brooks1982streamline,Hughes86a} method, the Galerkin Least Squares (GLS) method \cite{hughes1989new}, and the residual-based VMS formulation of Bazilevs et. al \cite{Bazilevs07b}.  In discretizations for which the pressure test space contains the divergence of every velocity trial solution, i.e., {\em divergence-conforming} discretizations, discrete mass conservation implies strong (pointwise) mass conservation \cite{John2017}.  Although loss of exact discrete and/or strong mass conservation is inconsequential for many applications, as demonstrated by our literature review, it is potentially-catastrophic in certain scenarios.  For instance, mass conservation is considered to be of prime importance for coupled flow-transport \cite{matthies07} and ``high pressure, low flow'' problems \cite{gerbeau97}.  Moreover, residual-based stabilization methods resembling the PSPG approach can fail spectacularly in conjunction with immersed boundary methods, especially when there is a large pressure jump across the immersed boundary \cite[Section 4.4.1]{Kamensky2015}.  The loss of mass conservation leads to a substantial effective leakage of fluid through boundaries, severely disrupting even the qualitative solution features.  This pathology can be avoided by using Galerkin's method with divergence-conforming function spaces \cite{Kamensky2017a,Casquero2018}, but at the cost of losing stability for high Reynolds numbers, leading \cite{Kamensky2017a} to resort to a low-order artificial diffusion in realistic problems.

To recover discrete mass conservation in existing residual-based stabilized formulations, it is possible to ignore pressure stabilization terms.  This is popular in practice when inf-sup stable velocity/pressure pairs are employed.  However, this usually comes with a loss of stability.  For instance, the reduced SUPG method which results from ignoring PSPG terms in the SUPG/PSPG method is not coercive with respect to an SUPG/PSPG-like norm.  As a consequence, an error analysis which shows robustness in the advection-dominated limit is missing for the reduced SUPG method \cite{gelhard2005stabilized}.  One can also improve mass conservation by employing grad-div stabilization with a large stabilization parameter, but this commonly results in locking of the velocity field and pressure oscillations except in specialized situations \cite{jenkins2014parameter}.

The present paper introduces a novel residual-based stabilization that can maintain discrete (and, for divergence-conforming methods, strong) mass conservation for inf-sup stable spaces with $H^1$-conforming pressure approximation, while providing optimal convergence in the diffusive regime, robustness in the advective regime, and energetic stability.  The gist of this method is to apply the VMS formalism using a Stokes projector for scale separation, as was recently suggested in \cite{ten2018correct}.  This yields a nonlinear algebraic system in which a discrete fine-scale pressure field must be solved for alongside the coarse-scale unknowns, such that the coarse- and fine-scale velocities separately satisfy discrete mass conservation.  Section \ref{sec:formulation} develops this VMS-based formulation for the Navier--Stokes problem and shows the formulation is energetically stable.  Section \ref{sec:oseen-conv} proves the convergence of the method, when applied to the linearized Navier--Stokes (Oseen) equations, assuming divergence-conforming discrete spaces.  Section \ref{sec:num-examples} demonstrates numerically that the convergence analysis can be extrapolated to the full Navier--Stokes problem and shows some initial results indicating that the proposed VMS-based formulation is effective as an LES methodology for turbulent incompressible flow.  We summarize our findings and outline future extensions of this work in Section \ref{sec:conclusions}.  

%{\color{red} Stolen from deleted section:} If Galerkin's method is applied to a stable pair of velocity and pressure spaces for which the range of the divergence operator applied to $\mathcal{V}^h$ is a subset of $\mathcal{Q}^h$, then the discrete solution is clearly solenoidal.  Such discretizations are referred to as {\em divergence-conforming} (or {\em div-conforming}).  An isogeometric example of such spaces is considered in \cite{Evans2011,Evans2013}, based on earlier work by Buffa et al. \cite{Buffa2010,Buffa2011}.  The divergence-free-velocity property is lost if the standard SUPG/PSPG stabilization following from previous VMS analyses is applied, due to the appearance of $\nabla q^h$ in the stabilization term.  However, with the formulation from the present work, $q^h$ participates only in the coarse-scale continuity equation, so the coarse-scale velocity, $\mathbf{u}^h$, remains exactly divergence-free.  We also note that the formulation \eqref{eq:coarse-problem-semidisc} simplifies somewhat with divergence-conforming spaces, because the various convection forms ($c$, $c_\text{cons}$, and $c_\text{skew}$) become equivalent.  In Remark \ref{rem:simplified-analysis} of the sequel, we outline how the convergence analysis of Section \ref{sec:oseen-conv} is also simplified.  

\section{Formulation for the Incompressible Navier--Stokes problem}\label{sec:formulation}
We begin by deriving our VMS-based formulation for the incompressible Navier--Stokes problem and showing that it is energetically stable.  We then show how to implement our method by statically condensing out the fine-scale velocity field after time discretization, and we present a simplified version of our method based on the concept of quasi-static subscales.
\subsection{Problem and notation}
The strong form of the incompressible Navier--Stokes problem with homogeneous Dirichlet boundary conditions and a Newtonian viscosity law is:  
$$
(S) \left\{ \hspace{5pt}
\parbox{6in}{
\noindent Find $\mathbf{u}:\overline{\Omega}\times\lbrack 0,T\rbrack\to\mathbb{R}^d$ and $p:\overline{\Omega}\times(0,T)\to\mathbb{R}$ such that
\begin{equation}
    \begin{array}{rclr}\partial_t\mathbf{u} + \mathbf{u}\cdot\nabla\mathbf{u} - \nabla\cdot\left(2\nu\nabla^s\mathbf{u}\right) + \nabla p &=& \mathbf{f} & \text{in }\Omega\times (0,T)\\
    \nabla\cdot\mathbf{u} &=& 0 & \text{in }\Omega\times(0,T)\\
    \mathbf{u}&=&\bm{0}&\text{on }\partial\Omega\times(0,T)\\
    \mathbf{u}\vert_{t=0} &=& \mathbf{u}_0&\text{in }\Omega\text{ ,}\end{array} \nonumber
\end{equation}
}
\right.
$$
Above, $\nu > 0$ is the kinematic viscosity, $\nabla^s$ is the symmetrized gradient, and $\mathbf{f}:\Omega\times(0,T)\to\mathbb{R}^d$ is a source term.  To facilitate the development of variational numerical approaches, we reformulate this problem in variaional form.  Let
\begin{align}
    \mathcal{V}&:=\left(H_0^1(\Omega)\right)^d\text{ ,}\\
        \mathcal{Q}&:=L^2_0(\Omega)\text{ ,}\\
        \mathcal{X}&:= \mathcal{V}\times\mathcal{Q}\text{ ,}
\end{align}
and
\begin{align}
    \mathcal{V}_T&:=\left\{\mathbf{v}\in C^0\left(\lbrack 0,T\rbrack;\mathcal{V}\right)~:~\partial_t\mathbf{v}\in L^2((0,T);\mathcal{V})\right\}\text{ ,}\\
        \mathcal{Q}_T&:=L^2((0,T);\Omega)\text{ ,}\\
        \mathcal{X}_T&:= \mathcal{V}_T\times\mathcal{Q}_T\text{ .}
\end{align}
We refer to $\mathcal{V}$ as the velocity space and and $\mathcal{Q}$ as the pressure space. The variational problem corresponding to Problem $(S)$ is then:
$$
(V) \left\{ \hspace{5pt}
\parbox{6in}{
\noindent Find $(\mathbf{u},p)\in\mathcal{X}_T$ such that $\mathbf{u}(0)=\mathbf{u}_0$ and, for a.e. $t\in (0,T)$,
\begin{equation}
    \left(\partial_t\mathbf{u}(t),\mathbf{v}\right)_{L^2(\Omega)} + c(\mathbf{u}(t),\mathbf{u}(t),\mathbf{v}) + k(\mathbf{u}(t),\mathbf{v}) - b(\mathbf{v},p(t)) + b(\mathbf{u}(t),q) = (\mathbf{f}(t),\mathbf{v})_{L^2(\Omega)} \nonumber
\end{equation}
for all $(\mathbf{v},q)\in\mathcal{X}$, where
\begin{align}
    c(\mathbf{v}_1,\mathbf{v}_2,\mathbf{v}_3) &= \int_{\Omega}(\mathbf{v}_1\cdot\nabla\mathbf{v}_2)\cdot\mathbf{v}_3\,d\Omega\text{ ,} \nonumber \\
    k(\mathbf{v}_1,\mathbf{v}_2) &= \int_\Omega 2\nu\nabla^s\mathbf{v}_1 : \nabla^s\mathbf{v}_2\,d\Omega\text{ ,} \nonumber \\
    b(\mathbf{v},q) &= \int_\Omega\nabla\cdot\mathbf{v} q\,d\Omega\text{ .} \nonumber
\end{align}
}
\right.
$$
In the sequel, when there is no risk of confusion with tuple notation, we may drop the subscript ``$L^2(\Omega)$'' from the $L^2(\Omega)$ inner product, and likewise for its induced norm.  

\subsection{Variational multiscale analysis}
Let $\mathcal{V}^h\subset\mathcal{V}$, $\mathcal{Q}^h\subset\mathcal{Q}$ be an inf-sup stable velocity-pressure pair \cite{Babuska71,Brezzi1974,Boffi2008}.  We refer to $\mathcal{V}^h$ as the coarse-scale velocity space and $\mathcal{Q}^h$ as the coarse-scale pressure space, and we assume that both $\mathcal{V}^h$ and $\mathcal{Q}^h$ are finite-dimensional.  We further assume throughout that $\mathcal{Q}^h$ is $H^1$-conforming, i.e., $\mathcal{Q}^h \subset H^1(\Omega)$.  Let $\mathcal{X}^h :=\mathcal{V}^h\times\mathcal{Q}^h$.  We define the action of the Stokes projector $\mathcal{P}_\mathcal{X}:\mathcal{X}\to\mathcal{X}^h$ as:  Given $\mathbf{X} = (\mathbf{w},r)\in\mathcal{X}$, find $\mathcal{P}_\mathcal{X}\mathbf{X} = (\mathbf{w}^h,r^h)\in\mathcal{X}^h$ such that
\begin{align}
    \nonumber&k(\mathbf{w}^h,\mathbf{v}^h) - b(\mathbf{v}^h,r^h) + b(\mathbf{w}^h,q^h) + (\tau_\text{C}\nabla\cdot\mathbf{w}^h,\nabla\cdot\mathbf{v}^h)\\
    &~~~= k(\mathbf{w},\mathbf{v}^h) - b(\mathbf{v}^h,r) + b(\mathbf{w},q^h) + (\tau_\text{C}\nabla\cdot\mathbf{w},\nabla\cdot\mathbf{v}^h)~~~~~~~\forall (\mathbf{v}^h,q^h)\in\mathcal{X}^h\text{ ,}
\end{align}
where the parameter $\tau_\text{C}\geq 0$ penalizes divergence of $\mathbf{w}^h$.  Let us decompose the solution of Problem $(V)$ at time $t$ into 
\begin{equation}
    (\mathbf{u}(t),p(t)) = (\mathbf{u}^h(t),p^h(t)) + (\mathbf{u}'(t),p'(t))\text{ ,}
\end{equation}
where 
\begin{equation}
    \mathcal{P}_\mathcal{X}(\mathbf{u}(t),p(t))\in\mathcal{X}^h\text{ .}
\end{equation}
Then we can decompose Problem $(V)$ into coarse-scale and fine-scale problems.  The coarse-scale problem is
\begin{align}
    \nonumber &(\partial_t\mathbf{u}^h(t),\mathbf{v}) + c(\mathbf{u}^h(t),\mathbf{u}^h(t),\mathbf{v}) + k(\mathbf{u}^h(t),\mathbf{v}^h) - b(\mathbf{v}^h,p^h(t)) + b(\mathbf{u}^h(t),q^h)\\
    \nonumber &+c(\mathbf{u}^h(t),\mathbf{u}'(t),\mathbf{v}^h)+c(\mathbf{u}'(t),\mathbf{u}^h(t),\mathbf{v}^h)+c(\mathbf{u}'(t),\mathbf{u}'(t),\mathbf{v}^h)+(\partial_t\mathbf{u}'(t),\mathbf{v}^h)\\
    &+ (\tau_\text{C}\nabla\cdot\mathbf{u}^h(t),\nabla\cdot\mathbf{v}^h) =(\mathbf{f}(t),\mathbf{v}^h)~~~~~~~\forall(\mathbf{v}^h,q^h)\in\mathcal{X}^h\label{eq:coarse-problem}
\end{align}
and the fine-scale problem is
\begin{align}
    \nonumber &(\partial_t\mathbf{u}'(t),\mathbf{v}') + c(\mathbf{u}^h(t),\mathbf{u}'(t),\mathbf{v}') + k(\mathbf{u}'(t),\mathbf{v}')-b(\mathbf{v}',p'(t)) + b(\mathbf{u}'(t),q')\\
    &+c(\mathbf{u}'(t),\mathbf{u}^h(t),\mathbf{v}') + c(\mathbf{u}'(t),\mathbf{u}'(t),\mathbf{v}') + b(\mathbf{u}^h(t),q') = -(\mathbf{r}_{\text{M}}(t),\mathbf{v}')~~~~~~~\forall(\mathbf{v}',q')\in\mathcal{X}'\text{ ,}\label{eq:fine-problem}
\end{align}
where $\mathbf{r}_{\text{M}}$ is the residual of the momentum balance equation from Problem $(S)$,
\begin{equation}\label{eq:residual}
    \mathbf{r}_\text{M} = \partial_t\mathbf{u}^h + \mathbf{u}^h\cdot\nabla\mathbf{u}^h - \nabla\cdot\left(2\nu\nabla^s\mathbf{u}^h\right) + \nabla p^h - \mathbf{f}\text{ ,}
\end{equation}
and $\mathcal{X}'$ is the orthogonal complement of $\mathcal{X}^h$ in $\mathcal{X}$, with respect to the projector $\mathcal{P}_\mathcal{X}$.  

\begin{remark}
The requirement that the coarse-scale pressure space $\mathcal{Q}^h$ be $H^1$-conforming cannot be relaxed for the VMS-based formulation presented in this paper, as we shall see later.  Inf-sup stable finite element and isogeometric spaces with continuous pressure approximation satisfy this requirement.  Thus, the VMS-based formulation presented here may be applied to the MINI element \cite{arnold1984stable}, Taylor-Hood element \cite{taylor1973numerical},  isogeometric Taylor-Hood and sub-grid elements \cite{bressan2013isogeometric}, and isogeometric N\'{e}d\'{e}lec and Raviart-Thomas elements \cite{buffa2011isogeometric}.  Isogeometric Raviart-Thomas elements are commonly referred to as divergence-conforming B-spline discretizations \cite{Evans2011}.
\end{remark}

\begin{remark}
The use of a Stokes projector is key to ensuring discrete mass conservation.  One can also enforce discrete mass conservation by defining the velocity field using constrained optimization, as in \cite{evans2009enforcement}, but this leads to a more complicated formulation.
\end{remark}

\begin{remark}
The Stokes projector employed here is similar to the Stokes projector presented in \cite{ten2018correct}.  However, the projector here has an additional grad-div term.  The appearance of grad-div stabilization in \eqref{eq:coarse-problem} follows from this grad-div term, and it {\em does not} represent a model of the fine-scale pressure, as it does in some other VMS-derived numerical approaches, e.g., \cite{Bazilevs07b}.  
\end{remark}

\begin{remark}
For standard $H^1$-conforming finite element velocity spaces $\mathcal{V}^h$, the term $\nabla\cdot\left(2\nu\nabla^s\mathbf{u}^h\right)$ appearing in \eqref{eq:residual} is not square-integrable, and hence the inner product $(\mathbf{r}_{\text{M}},\mathbf{v}')$ appearing in \eqref{eq:fine-problem} is not well-defined.  To remedy this situation, we replace the inner-product $(\mathbf{r}_{\text{M}},\mathbf{v}')$ with a corresponding summation of inner-products over element interiors, i.e., $\sum_{e=1}^{n_{el}} (\mathbf{r}_{\text{M}},\mathbf{v}')_{\Omega^e}$.  This is standard practice with stabilized and multiscale finite element methods \cite{HugScoFr04}.  To simplify exposition, we overload the notation $(\mathbf{r}_{\text{M}},\mathbf{v}')$ to also mean this summation of inner-products over element interiors when the momentum residual $\mathbf{r}_{\text{M}}$ is not square-integrable.  One may alternatively globally reconstruct $\nabla\cdot\left(2\nu\nabla^s\mathbf{u}^h\right)$, which yields improved accuracy in certain settings \cite{jansen1999better}.
\end{remark}

\subsection{Treatment of convection terms}
For any $\mathbf{a}\in\mathcal{V}$ with $\nabla\cdot\mathbf{a}= 0$,
\begin{equation}
    c(\mathbf{a},\mathbf{v}_1,\mathbf{v}_2) = c_{\text{cons}}(\mathbf{a},\mathbf{v}_1,\mathbf{v}_2) = c_{\text{skew}}(\mathbf{a},\mathbf{v}_1,\mathbf{v}_2)~~~~~~~\forall\mathbf{v}_1,\mathbf{v}_2\in\mathcal{V}\text{ ,}
\end{equation}
where
\begin{align}
    c_{\text{cons}}(\mathbf{a},\mathbf{v}_1,\mathbf{v}_2) &= -\int_\Omega\mathbf{v}_1\cdot\left(\mathbf{a}\cdot\nabla\mathbf{v}_2\right)\,d\Omega\text{ ,}\\
    c_{\text{skew}}(\mathbf{a},\mathbf{v}_1,\mathbf{v}_2) &= \frac{1}{2}\left(c(\mathbf{a},\mathbf{v}_1,\mathbf{v}_2) + c_{\text{cons}}(\mathbf{a},\mathbf{v}_1,\mathbf{v}_2)\right)\text{ .}
\end{align}
Thus, in \eqref{eq:coarse-problem}, we can select
\begin{align}
    c(\mathbf{u}^h(t)+\mathbf{u}'(t),\mathbf{u}^h(t),\mathbf{v}^h) &= c_{\text{skew}}(\mathbf{u}^h(t)+\mathbf{u}'(t),\mathbf{u}^h(t),\mathbf{v}^h)\text{ ,}\\
    c(\mathbf{u}^h(t) + \mathbf{u}'(t),\mathbf{u}'(t),\mathbf{v}^h) &= c_{\text{cons}}(\mathbf{u}^h(t) + \mathbf{u}'(t),\mathbf{u}'(t),\mathbf{v}^h)\text{ .}
\end{align}
These selections are critical for energy stability.

\subsection{Modifying the fine-scale problem}
To adapt \eqref{eq:coarse-problem}--\eqref{eq:fine-problem} into a practical numerical method, we make the following modeling choices:
\begin{enumerate}
    \item We assume that $c(\mathbf{u}'(t),\mathbf{u}'(t),\mathbf{v}') = 0$, i.e., we ignore fine-scale nonlinearity in the fine-scale problem.
    \item We assume that
    $b(\mathbf{u}^h(t),q') = 0$, i.e., we ignore compressibility of the coarse-scale velocity field in the fine-scale problem.
    \item We introduce the model
    \begin{equation}
        c(\mathbf{u}^h(t),\mathbf{u}'(t),\mathbf{v}') + k(\mathbf{u}'(t),\mathbf{v}') = (\tau_{\text{M}}^{-1}\mathbf{u}'(t),\mathbf{v}')\text{ ,}
    \end{equation}
    where $\tau_{\text{M}}$ is a stabilization parameter (whose value remains to be determined).  This is inspired by analogy to advection--diffusion, where the fine-scale Green's function is localized in the case of an $H^1_0$ projector \cite{HugSan06}.  
    \item We replace $\mathcal{X}'$ with surrogate fine-scale spaces $\mathcal{V}'\times\mathcal{Q}'$.  In particular, we select
    \begin{align}
        \mathcal{V}'&:=\left(L^2(\Omega)\right)^d\text{ ,}\label{eq:v-prime-l2}\\
        \mathcal{Q}'&\supseteq\mathcal{Q}^h\text{ .}
    \end{align}
    We further assume that $\mathcal{Q}'$ is finite-dimensional.
\end{enumerate}
Note that since $\mathcal{V}' \not\subseteq H(\text{div},\Omega)$, we must employ the divergence theorem to replace $b(\mathbf{v}',p'(t))$ and $b(\mathbf{u}'(t),q')$ by $-(\nabla p'(t),\mathbf{v}')$ and $-(\nabla q',\mathbf{u}'(t))$, respectively, in \eqref{eq:fine-problem}.  This requires that $\mathcal{Q}'\subset H^1(\Omega)$, and since $\mathcal{Q}^h \subseteq \mathcal{Q}'$, this also requires $\mathcal{Q}^h\subset H^1(\Omega) $, which we have already assumed.

\subsection{Semi-discrete formulation}\label{sec:semi-discrete}
Define
\begin{align}
    \mathcal{V}^h_T&:=\left\{\mathbf{v}\in C^0(\lbrack 0,T\rbrack;\mathcal{V}^h)~:~\partial_t\mathbf{v}\in L^2((0,T);\mathcal{V}^h)\right\}\text{ ,}\\
    \mathcal{V}'_T&:=\left\{\mathbf{v}\in C^0(\lbrack 0,T\rbrack;\mathcal{V}')~:~\partial_t\mathbf{v}\in L^2((0,T);\mathcal{V}')\right\}\text{ ,}\\
    \mathcal{Q}^h_T&:= L^2((0,T);\mathcal{Q}^h)\text{ ,}\\
    \mathcal{Q}'_T&:= L^2((0,T);\mathcal{Q}')\text{ .}
\end{align}
We may now define the following semi-discrete formulation:  
$$
(V^h) \left\{ \hspace{5pt}
\parbox{5.905in}{
\noindent Find $(\mathbf{u}^h,p^h)\in\mathcal{V}_T^h\times\mathcal{Q}_T^h$ and $(\mathbf{u}',p')\in\mathcal{V}_T'\times\mathcal{Q}_T'$ satisfying $\mathbf{u}^h(0)=\mathbf{u}^h_0$, $\mathbf{u}'(0)=\mathbf{u}'_0$, the coarse-scale problem
\begin{align}
    \nonumber &(\partial_t\mathbf{u}^h(t),\mathbf{v}^h) + c_{\text{skew}}(\mathbf{u}^h(t),\mathbf{u}^h(t),\mathbf{v}^h) + k(\mathbf{u}^h(t),\mathbf{v}^h) - b(\mathbf{v}^h,p^h(t)) + b(\mathbf{u}^h(t),q^h)\\
    \nonumber &+c_{\text{cons}}(\mathbf{u}^h(t),\mathbf{u}'(t),\mathbf{v}^h)+c_{\text{skew}}(\mathbf{u}'(t),\mathbf{u}^h(t),\mathbf{v}^h)+c_{\text{cons}}(\mathbf{u}'(t),\mathbf{u}'(t),\mathbf{v}^h)+(\partial_t\mathbf{u}'(t),\mathbf{v}^h)\\
    &+ (\tau_\text{C}\nabla\cdot\mathbf{u}^h(t),\nabla\cdot\mathbf{v}^h) =(\mathbf{f}(t),\mathbf{v}^h) \nonumber
\end{align}
for all $(\mathbf{v}^h,q^h)\in\mathcal{X}^h$ and a.e. $t\in (0,T)$, and the fine-scale problem
\begin{align}
    \nonumber & (\partial_t\mathbf{u}'(t),\mathbf{v}') + (\tau^{-1}_{\text{M}}\mathbf{u}'(t),\mathbf{v}') + c(\mathbf{u}'(t),\mathbf{u}^h(t),\mathbf{v}')\\
    &+(\nabla p'(t),\mathbf{v}') - (\nabla q',\mathbf{u}'(t)) = -(\mathbf{r}_\text{M}(t),\mathbf{v}') \nonumber
\end{align}
for all $(\mathbf{v}',q')\in\mathcal{V}'\times\mathcal{Q}'$ and a.e. $t\in (0,T)$.
}
\right.
$$
Above, $\mathbf{u}^h_0$ and $\mathbf{u}'_0$ are suitable initial conditions for the coarse-scale and fine-scale velocity fields which may be defined, for instance, using the Stokes projector.  We have the following trivial result for the semi-discrete formulation given by Problem $(V^h)$ since $\mathcal{Q}^h \subseteq \mathcal{Q}'$.

\begin{lemma}[Discrete Mass Conservation]\label{mass}
It holds that $(q^h,\nabla \cdot \mathbf{u}^h(t))=0$ and $-(\nabla q^h,\mathbf{u}'(t))=0$
for all $q^h \in \mathcal{Q}^h$ and a.e. $t\in (0,T)$.
\end{lemma}

Lemma \ref{mass} states both the coarse-scale and fine-scale velocity fields are discretely divergence-free, and hence the semi-discrete formulation given by Problem $(V^h)$ has not upset the underlying mass conservation properties of Galerkin's method.  In particular, if the coarse-scale spaces $\mathcal{V}^h$ and $\mathcal{Q}^h$ are divergence-conforming, i.e., $\nabla \cdot \mathcal{V}^h \subseteq \mathcal{Q}^h$, then the formulation yields a pointwise divergence-free coarse-scale velocity field.  This is due to the fact discretely divergence-free coarse-scale velocity fields are pointwise divergence-free for divergence-conforming discretizations \cite{Scott1985}.  As discussed in Section \ref{sec:introduction}, this property is not shared by the residual-based VMS formulation of Bazilevs et. al \cite{Bazilevs07b}.

\begin{remark}
The semi-discrete formulation given by Problem $(V^h)$ is similar to the Galerkin/least-squares formulation with dynamic divergence-free small-scales (GLSDD) presented in \cite{ten2018correct}.  However, there are a few key differences.  First of all, GLSDD is of Galerkin/least-squares (GLS) type, while the formulation presented here is not.  Second of all, the fine-scale pressure space in GLSDD is taken to be identically the coarse-scale pressure space.  Finally, GLSDD is specifically for divergence-conforming discretizations.
\end{remark}

\subsection{Energetic stability}
We proceed with an energetic analysis of the semi-discrete formulation given by Problem $(V^h)$.  To do so, we make a number of assumptions.  We first assume that $\mathcal{V}^h$ and $\mathcal{Q}^h$ are (mapped) piecewise-polynomial or piecewise-rational approximation spaces with respect to a mesh $\mathcal{M}$.  We denote each element of the mesh $\mathcal{M}$ as $\Omega^e$, and we note that $\Omega = \text{int}\left(\cup_{e=1}^{n_{el}} \overline{\Omega^e} \right)$ where $n_{el}$ is the number of elements in the mesh.  We next make two assumptions regarding the velocity space $\mathcal{V}^h$ and the stabilization parameter $\tau_{\text{M}}$.

\begin{assumption}\label{ass1}
For each element $\Omega^e \in \mathcal{M}$, there exists a constant $C_\text{inv} \geq 0$ independent of the diameter $h_e$ of $\Omega^e$ such that
\begin{equation}
     \left(\nabla\cdot(2\nu\nabla^s\mathbf{v}^h),\nabla\cdot(2\nu\nabla^s\mathbf{v}^h)\right)_{\Omega^e}\leq \frac{C_{\text{inv}}}{h^2_e}\left(\nu\nabla^s\mathbf{v}^h,\nu\nabla^s\mathbf{v}^h\right)_{\Omega^e}
\end{equation}
for all $\mathbf{v}^h \in \mathcal{V}^h$.
\end{assumption}

\begin{assumption}\label{ass2}
For each element $\Omega^e \in \mathcal{M}$,
\begin{equation}
     \tau_\text{M} \leq \frac{h_e^2}{C_\text{inv} \nu}\text{ ,}
\end{equation}
where $C_\text{inv}$ is the constant associated with Assumption \ref{ass1}.
\end{assumption}

Assumption \ref{ass1} holds for standard finite element and isogeometric spatial discretizations \cite{Bazilevs2006}, while Assumption \ref{ass2} holds for standard choices of the stabilization parameter \cite{HugScoFr04}.  With the above assumptions in hand, we have the following energetic stability result for our formulation.

\begin{lemma}[Energetic Stability]\label{energy}
Provided that Assumptions \ref{ass1} and \ref{ass2} are satisfied, the kinetic energy associated with Problem $(V^h)$ evolves as
\begin{align}
    \frac{d}{dt}\left(\frac{1}{2}\left\Vert\mathbf{u}^h + \mathbf{u}'\right\Vert^2\right) \leq (\mathbf{f},\mathbf{u}^h+\mathbf{u}') - \frac{1}{2}k(\mathbf{u}^h,\mathbf{u}^h) - (\tau_\text{C}\nabla\cdot\mathbf{u}^h,\nabla\cdot\mathbf{u}^h) - \frac{1}{2}(\tau_\text{M}^{-1}\mathbf{u}',\mathbf{u}')\text{ .}
\end{align}
In particular, the kinetic energy decays in time for a homogeneous source term $\mathbf{f} = \bm{0}$.
\end{lemma}
\begin{proof}
Let $\mathbf{v}^h = \mathbf{u}^h(t)$, $\mathbf{v}' = \mathbf{u}'(t)$, $q^h = p^h(t)$, $q' = p'(t)$ in Problem $(V^h)$.  Summing the coarse-scale and fine-scale problems yields
\begin{align}
    \nonumber & (\partial_t(\mathbf{u}^h + \mathbf{u}')(t),(\mathbf{u}^h + \mathbf{u}')(t)) + c_\text{skew}(\mathbf{u}^h(t),\mathbf{u}^h(t),\mathbf{u}^h(t))\\
    \nonumber & +c_\text{skew}(\mathbf{u}'(t),\mathbf{u}^h(t),\mathbf{u}^h(t)) + c_\text{cons}(\mathbf{u}^h(t),\mathbf{u}'(t),\mathbf{u}^h(t))\\
    \nonumber & +c(\mathbf{u}^h(t),\mathbf{u}^h(t),\mathbf{u}'(t)) + c_\text{cons}(\mathbf{u}'(t),\mathbf{u}'(t),\mathbf{u}^h(t))\\
    \nonumber & +c(\mathbf{u}'(t),\mathbf{u}^h(t),\mathbf{u}'(t)) + k(\mathbf{u}^h(t),\mathbf{u}^h(t)) + (\tau_\text{C}\nabla\cdot\mathbf{u}^h(t),\nabla\cdot\mathbf{u}^h(t))\\
     & +(\tau^{-1}_\text{M}\mathbf{u}'(t),\mathbf{u}(t)) + (\nabla p^h(t),\mathbf{u}'(t)) - (\nabla\cdot(2\nu\nabla^s\mathbf{u}^h(t)),\mathbf{u}'(t)) = (\mathbf{f},\mathbf{u}^h(t) + \mathbf{u}'(t))\text{ .}\label{eq:energy-big-eqn}
\end{align}
Note that
\begin{align}
    (\partial_t(\mathbf{u}^h + \mathbf{u}')(t),(\mathbf{u}^h + \mathbf{u}')(t)) &= \frac{d}{dt}\left(\frac{1}{2}\left\Vert\mathbf{u}^h + \mathbf{u}'\right\Vert^2\right)(t)\text{ ,}\label{eq:first-id}\\
    c_\text{skew}(\mathbf{u}^h(t),\mathbf{u}^h(t),\mathbf{u}^h(t)) &= c_\text{skew}(\mathbf{u}'(t),\mathbf{u}^h(t),\mathbf{u}^h(t)) = 0\text{ ,}\\
    c_\text{cons}(\mathbf{u}^h(t),\mathbf{u}'(t),\mathbf{u}^h(t)) &+ c(\mathbf{u}^h(t),\mathbf{u}^h(t),\mathbf{u}'(t)) = 0\text{ ,}\\
    c_\text{cons}(\mathbf{u}'(t),\mathbf{u}'(t),\mathbf{u}^h(t)) &+ c(\mathbf{u}'(t),\mathbf{u}^h(t),\mathbf{u}'(t)) = 0\text{ ,}\\
    (\nabla p^h(t),\mathbf{u}'(t)) &= 0\text{ .}\label{eq:last-id}
\end{align}
Observe that \eqref{eq:last-id} holds due to Lemma \ref{mass}.  Using \eqref{eq:first-id}--\eqref{eq:last-id} in \eqref{eq:energy-big-eqn}, we get
\begin{align}
    \nonumber\frac{d}{dt}\left(\frac{1}{2}\left\Vert\mathbf{u}^h + \mathbf{u}'\right\Vert^2\right) =~& (\mathbf{f},\mathbf{u}^h + \mathbf{u}') + (\nabla\cdot(2\nu\nabla^s\mathbf{u}),\mathbf{u}') \\
    &- \left(k(\mathbf{u}^h,\mathbf{u}^h) + (\tau_\text{C}\nabla\cdot\mathbf{u}^h,\nabla\cdot\mathbf{u}^h) + (\tau_\text{M}^{-1}\mathbf{u}',\mathbf{u}')\right)\text{ .}
\end{align}
Using Young's inequality,
\begin{align}
    (\nabla\cdot(2\nu\nabla^s\mathbf{u}^h),\mathbf{u}')\leq \frac{1}{2}\left(\tau_\text{M}\nabla\cdot(2\nu\nabla^s\mathbf{u}^h),\nabla\cdot(2\nu\nabla^s\mathbf{u}^h)\right) + \frac{1}{2}\left(\tau_\text{M}^{-1}\mathbf{u}',\mathbf{u}'\right)\text{ .}
\end{align}
Since
\begin{equation}
    \tau_\text{M}\leq\frac{h_e^2}{C_\text{inv}\nu}~~~\text{on }\Omega^e\text{ ,}
\end{equation}
it holds that
\begin{equation}
    \left(\tau_\text{M} \nabla\cdot(2\nu\nabla^s\mathbf{u}^h),\nabla\cdot(2\nu\nabla^s\mathbf{u}^h)\right)_{\Omega^e}\leq \left(\nu\nabla^s\mathbf{u}^h,\nabla^s\mathbf{u}^h\right)_{\Omega^e}\text{ .}
\end{equation}
Consequently we have
\begin{equation}
    (\nabla\cdot(2\nu\nabla^s\mathbf{u}^h),\mathbf{u}') \leq \frac{1}{2}k(\mathbf{u}^h,\mathbf{u}^h) + \frac{1}{2}(\tau_\text{M}^{-1}\mathbf{u}',\mathbf{u}')
\end{equation}
and kinetic energy evolves as
\begin{align}
    \frac{d}{dt}\left(\frac{1}{2}\left\Vert\mathbf{u}^h + \mathbf{u}'\right\Vert^2\right) \leq (\mathbf{f},\mathbf{u}^h+\mathbf{u}') - \frac{1}{2}k(\mathbf{u}^h,\mathbf{u}^h) - (\tau_\text{C}\nabla\cdot\mathbf{u}^h,\nabla\cdot\mathbf{u}^h) - \frac{1}{2}(\tau_\text{M}^{-1}\mathbf{u}',\mathbf{u}')\text{ .}
\end{align}
This is the desired result.
\end{proof}

\begin{remark}
Note that Lemma \ref{energy} states that the total kinetic energy, that is, the kinetic energy associated with both the coarse- and fine-scales, decays in time for a homogeneous source term.  It does not state that the kinetic energy associated with just the coarse-scales decays in time for such a setting.  Indeed, we expect our VMS-based formulation to allow for backscatter of energy from fine-scales to coarse-scales \cite{principe2010dissipative}.
\end{remark}

\subsection{Time discretization and static condensation of the fine-scale velocity field}\label{sec:timediscretization}
Problem $(V^h)$ constitutes a set of differential-algebraic equations for the unknown coarse-scale and fine-scale velocity and pressure fields, and thus it can be discretized in time using virtually any time discretization scheme for differential-algebraic equations.  Once Problem $(V^h)$ is discretized in time with such a scheme, the fine-scale velocity field may be statically condensed from the system at each time-step, resulting in a nonlinear algebraic system for the coarse-scale velocity field, coarse-pressure pressure field, and  fine-scale pressure field.  To illustrate this, suppose that Problem $(V^h)$ is discretized in time using the backward Euler scheme.  Then, we have the following problem at the $n^\text{th}$ time step $t_n$: Given $(\mathbf{u}^h_{n-1},p^h_{n-1})\in\mathcal{V}^h \times\mathcal{Q}^h$ and $(\mathbf{u}'_{n-1},p'_{n-1})\in\mathcal{V}' \times\mathcal{Q}'$, find $(\mathbf{u}^h_n,p^h_n)\in\mathcal{V}^h \times\mathcal{Q}^h$ and $(\mathbf{u}'_n,p'_n)\in\mathcal{V}' \times\mathcal{Q}'$ satisfying the coarse-scale problem
\begin{align}
    \nonumber &\left( \frac{ \mathbf{u}^h_n - \mathbf{u}^h_{n-1} }{\Delta t_n} ,\mathbf{v}^h \right) + c_{\text{skew}}(\mathbf{u}^h_n,\mathbf{u}^h_n,\mathbf{v}^h) + k(\mathbf{u}^h_n,\mathbf{v}^h) - b(\mathbf{v}^h,p^h_n) + b(\mathbf{u}^h_n,q^h)\\
    \nonumber &+c_{\text{cons}}(\mathbf{u}^h_n,\mathbf{u}'_n,\mathbf{v}^h)+c_{\text{skew}}(\mathbf{u}'_n,\mathbf{u}^h_n,\mathbf{v}^h)+c_{\text{cons}}(\mathbf{u}'_n,\mathbf{u}'_n,\mathbf{v}^h)+\left( \frac{ \mathbf{u}'_n - \mathbf{u}'_{n-1} }{\Delta t_n} ,\mathbf{v}^h \right)\\
    &+ (\tau_\text{C}\nabla\cdot\mathbf{u}^h_n,\nabla\cdot\mathbf{v}^h) =(\mathbf{f}(t_n),\mathbf{v}^h)
    \label{eq:precondensed_coarse}
\end{align}
for all $(\mathbf{v}^h,q^h)\in\mathcal{X}^h$ and the fine-scale problem
\begin{align}
    &\left( \frac{ \mathbf{u}'_n - \mathbf{u}'_{n-1} }{\Delta t_n} ,\mathbf{v}' \right) + (\tau^{-1}_{\text{M}}\mathbf{u}'_n,\mathbf{v}') + c(\mathbf{u}'_n,\mathbf{u}^h_n,\mathbf{v}') \nonumber \\
    &+(\nabla p'_n,\mathbf{v}') - (\nabla q',\mathbf{u}'_n) = -(\mathbf{r}_\text{M}(t_n),\mathbf{v}') \label{eq:precondensed_fine}
\end{align}
for all $(\mathbf{v}^h,q^h)\in\mathcal{V}'\times\mathcal{Q}'$ where $\Delta t_n = t_n - t_{n-1}$.  Setting $(\mathbf{v}',q') = (\mathbf{v}',0)$ in \eqref{eq:precondensed_fine} yields a nonlinear algebraic equation for the fine-scale velocity field $\mathbf{u}'_{n}$.  The solution of this equation is \begin{equation}
\mathbf{u}'_{n} =  \left( \left( 1 + \frac{\Delta t_{n}}{\tau_{\text{M}}} \right) \textbf{I} + \Delta t_{n} \nabla \textbf{u}^h_{n} \right)^{-1} \left(\mathbf{u}'_{n-1} - \Delta t_{n} \left( \nabla p'_{n} + \mathbf{r}_\text{M}(t_{n}) \right) \right)\text{ ,}
\label{eq:static_cond}
\end{equation}
where $\textbf{I}$ is the identity matrix.  Inserting \eqref{eq:static_cond} back into \eqref{eq:precondensed_coarse} and \eqref{eq:precondensed_fine} then yields a nonlinear algebraic system for the coarse-scale velocity field $\mathbf{u}^h_{n}$, coarse-pressure pressure field $p^h_{n}$, and fine-scale pressure field $p'_{n}$.  Note in particular that inserting \eqref{eq:static_cond} into \eqref{eq:precondensed_fine} with $(\mathbf{v}',q') = (\mathbf{0},q')$ yields
\begin{align}
\left(\nabla q', \Delta t_{n} \mathbf{K}_\text{n} \nabla p'_n\right) = -\left(\nabla q', \Delta t_{n} \mathbf{K}_\text{n} \mathbf{r}_\text{M}(t_{n})\right) + \left(\nabla q', \mathbf{K}_\text{n}  \mathbf{u}'_{n-1}\right)\text{ ,}
\end{align}
where $\mathbf{K}_\text{n} = \left( \left( 1 + \frac{\Delta t_{n}}{\tau_{\text{M}}} \right) \textbf{I} + \Delta t_{n} \nabla \textbf{u}^h_{n} \right)^{-1}$.  Thus we see the fine-scale mass conservation constraint yields a Poisson-like problem for the fine-scale pressure field after static condensation of the fine-scale velocity field.  Alternative time discretization schemes, including the popular generalized $\alpha$ scheme \cite{Jansen2000}, yield similar Poisson-like problems for the fine-scale pressure field.

\subsection{Quasi-static subscales}\label{sec:alt-form}
The semi-discrete formulation given by Problem $(V^h)$ is analogous to the method of dynamics subscales proposed by Codina et al. \cite{CoPrGuBa07}.  By making other modeling choices, we can arrive at other methods.  In particular, we can arrive at an analogous model to the quasi-static subscale model of Bazilevs et al. \cite{Bazilevs07b}, by omitting the fine-scale unsteadiness and cross-stress terms from the fine-scale problem in Problem $(V^h)$, allowing us to solve for the fine-scale velocity directly:
\begin{equation}
    \mathbf{u}' = -\tau_{\text{M}}\left(\nabla p' + \mathbf{r}_\text{M}\right)\text{ .}
\end{equation}
(Note that for full consistency with \cite{Bazilevs07b}, one would also write every convection term in the coarse-scale problem in the conservative form $c_{\text{cons}}$.)  The resulting model is not guaranteed to be energy-stable, but it has been successfully applied in many turbulence studies and challenging engineering problems, e.g., \cite{Bazilevs07b,Bazilevs08a,Akkerman08a,Bazilevs09d,Bazilevs10d,Takizawa11n,BazHsu12a,BazHsu12b,HsuAkk12b,Korobenko14a}.  

\section{Convergence analysis for the Oseen problem}\label{sec:oseen-conv}
We continue with a convergence analysis of our VMS-based formulation.  To do so, we consider a simplified steady linear problem---Oseen flow---with a fixed solenoidal advection velocity, $\mathbf{a}\in\mathcal{V}$.  In weak form, it is:  Find $\mathbf{u}\in\mathcal{V}$ and $p\in\mathcal{Q}$ such that, for all $\mathbf{v}\in\mathcal{V}$ and $q\in\mathcal{Q}$,
\begin{equation}\label{eq:oseen-weak}
    c(\mathbf{a},\mathbf{u},\mathbf{v}) + k(\mathbf{u},\mathbf{v}) - b(\mathbf{v},p) + b(\mathbf{u},q) = (\mathbf{f},\mathbf{v})\text{ .}
\end{equation}
A discretization for \eqref{eq:oseen-weak} analogous to Problem $(V^h)$ is:  Find $(\mathbf{u}^h,p^h,\mathbf{u}',p')\in\mathcal{V}^h\times\mathcal{Q}^h\times\mathcal{V}'\times\mathcal{Q}'$ such that, for all $(\mathbf{v}^h,q^h,\mathbf{v}',q')\in\mathcal{V}^h\times\mathcal{Q}^h\times\mathcal{V}'\times\mathcal{Q}'$,
\begin{equation}
    A\left((\mathbf{u}^h,p^h,\mathbf{u}',p'),(\mathbf{v}^h,q^h,\mathbf{v}',q')\right) = (\mathbf{f},\mathbf{v}^h+\mathbf{v}')\text{ ,} \label{eq:full}
\end{equation}
where
\begin{align}
    \nonumber & A\left((\mathbf{u}^h,p^h,\mathbf{u}',p'),(\mathbf{v}^h,q^h,\mathbf{v}',q')\right) \\
    \nonumber &~~~ =c(\mathbf{a},\mathbf{u}^h,\mathbf{v}^h) + k(\mathbf{u}^h,\mathbf{v}^h) - b(\mathbf{v}^h,p^h) + b(\mathbf{u}^h,q^h)\\
    \nonumber &~~~ +c_\text{cons}(\mathbf{a},\mathbf{u}',\mathbf{v}^h) + (\tau_\text{C}\nabla\cdot\mathbf{u}^h,\nabla\cdot\mathbf{v}^h) + (\tau_\text{M}^{-1}\mathbf{u}',\mathbf{v}')\\
    \nonumber&~~~-b(\mathbf{v}',p') + b(\mathbf{u}',q') + (\mathbf{a}\cdot\nabla\mathbf{u}^h,v')\\
    &~~~-(\nabla\cdot(2\nu\nabla^s\mathbf{u}^h,\mathbf{v}') + (\nabla p^h,\mathbf{v}')\text{ .}
\end{align}
The simplified treatment of advection relative to Problem $(V^h)$ is a result of $\nabla\cdot\mathbf{a}=0$.

As in the alternative formulation of Section \ref{sec:alt-form}, we can formally eliminate $\mathbf{u}'$ from the above formulation:
\begin{equation}
    \mathbf{u}' = \tau_\text{M}\left(\mathbf{f} - \mathbf{a}\cdot\nabla\mathbf{u}^h + \nabla\cdot(2\nu\nabla^s\mathbf{u}^h) - \nabla p^h - \nabla p'\right)\text{ .}
\end{equation}
We can also replace the fine-scale pressure field $p'$ in the formulation with the total pressure $\tilde{p} = p^h + p' \in \tilde{\mathcal{Q}}$ where $\tilde{\mathcal{Q}} := \mathcal{Q}^h + \mathcal{Q}'$.  Note that since $\mathcal{Q}^h \subseteq \mathcal{Q}'$, $\tilde{\mathcal{Q}} = \mathcal{Q}'$, but we elect to use the notation $\tilde{\mathcal{Q}}$ in order to distinguish between total pressure and fine-scale pressure.  Finally, we can eliminate the coarse-scale pressure field $p^h$ by restricting the coarse-scale velocity trial and test functions to the weakly divergence-free subspace
$$\mathring{\mathcal{V}}^h := \left\{ \mathbf{v}^h \in \mathcal{V}^h : b(\mathbf{v}^h,q^h) =0 \text{ for all } q^h \in \mathcal{Q}^h \right\}.$$
The aforementioned modifications result in the following reduced formulation:  Find $(\mathbf{u}^h,\tilde{p})\in\mathring{\mathcal{V}}^h\times\tilde{\mathcal{Q}}$ such that, for all $(\mathbf{v}^h,\tilde{q})\in\mathring{\mathcal{V}}^h\times\tilde{\mathcal{Q}}$,
\begin{equation}
    A_\text{red}\left((\mathbf{u}^h,\tilde{p}),(\mathbf{v}^h,\tilde{q})\right) = (\mathbf{f},\mathbf{v}^h) + \left(\mathbf{f}, \tau_\text{M} (\mathbf{a}\cdot\nabla\mathbf{v}^h + \nabla \tilde{q})\right)\text{ ,} \label{eq:reduced}
\end{equation}
where
\begin{align}
    \nonumber &A_\text{red}\left((\mathbf{u}^h,\tilde{p}),(\mathbf{v}^h,\tilde{q})\right)\\
    \nonumber&~~~=c(\mathbf{a},\mathbf{u}^h,\mathbf{v}^h) + k(\mathbf{u}^h,\mathbf{v}^h)
    +(\tau_\text{C}\nabla\cdot\mathbf{u}^h,\nabla\cdot\mathbf{v}^h)
    +\left(\mathbf{a}\cdot\nabla\mathbf{u}^h - \nabla\cdot(2\nu\nabla^s\mathbf{u}^h) + 
    \nabla \tilde{p},\tau_\text{M}(\mathbf{a}\cdot\nabla\mathbf{v}^h + \nabla \tilde{q})\right)\text{ .}
\end{align}
While the coarse-scale and fine-scale pressure fields $p^h$ and $p'$ serve as Lagrange multipliers associated with the mass conservation constraint in the full formulation given by \eqref{eq:full}, the total pressure field $\tilde{p}$ is determined from momentum conservation via a Poisson-like problem in the above reduced formulation.

To proceed forward, we make a fundamental assumption regarding the coarse-scale spaces $\mathcal{V}^h$ and $\mathcal{Q}^h$.

\begin{assumption}\label{ass3}
The coarse-scale spaces $\mathcal{V}^h$ and $\mathcal{Q}^h$ are divergence-conforming, i.e., $\nabla \cdot \mathcal{V}^h \subseteq \mathcal{Q}^h$.
\end{assumption}

Under the above assumption, discretely divergence-free coarse-scale velocity fields are pointwise divergence-free.  That is,
$$\mathring{\mathcal{V}}^h \equiv \left\{ \mathbf{v}^h \in \mathcal{V}^h : \nabla \cdot \mathbf{v}^h = 0 \right\}.$$
The following consistency result holds under such an assumption.

\begin{lemma}[Consistency]\label{consistency}
Provided that Assumption \ref{ass3} is satisfied and the exact solution $(\mathbf{u},p)\in \mathcal{V} \times \mathcal{Q}$ to the Oseen problem lies in $(H^2(\Omega))^d \times H^1(\Omega)$, it holds that
\begin{equation}
    A_\text{red}\left((\mathbf{u} - \mathbf{u}^h,p - \tilde{p}),(\mathbf{v}^h,\tilde{q})\right) = 0
\end{equation}
for all $(\mathbf{v}^h,\tilde{q})\in\mathring{\mathcal{V}}^h\times\mathcal{Q}'$.
\end{lemma}
\begin{proof}
If the solution $(\mathbf{u},p)\in \mathcal{V} \times \mathcal{Q}$ lies in $(H^2(\Omega))^d \times H^1(\Omega)$, it holds that
\begin{equation}
    A_\text{red}\left((\mathbf{u},p),(\mathbf{v},q)\right) = (\mathbf{f},\mathbf{v}) + \left(\mathbf{f}, \tau_\text{M} (\mathbf{a}\cdot\nabla\mathbf{v} + \nabla q)\right)
\end{equation}
for all $(\mathbf{v},q) \in \mathring{\mathcal{V}} \times \mathcal{Q}$ where $\mathring{\mathcal{V}} := \left\{ \mathbf{v} \in \mathcal{V} : \nabla \cdot \mathbf{v} = 0 \right\}$.  Since Assumption \ref{ass3} is satisfied, $\mathring{\mathcal{V}}^h \subset \mathring{\mathcal{V}}$, and hence
\begin{equation}
    A_\text{red}\left((\mathbf{u},p),(\mathbf{v}^h,\tilde{q})\right) = (\mathbf{f},\mathbf{v}^h) + \left(\mathbf{f}, \tau_\text{M} (\mathbf{a}\cdot\nabla\mathbf{v}^h + \nabla \tilde{q})\right) \label{eq:cons1}
\end{equation}
for all $(\mathbf{v}^h,\tilde{q})\in\mathring{\mathcal{V}}^h\times\tilde{\mathcal{Q}}$.  By construction,
\begin{equation}
    A_\text{red}\left((\mathbf{u}^h,\tilde{p}),(\mathbf{v}^h,\tilde{q})\right) = (\mathbf{f},\mathbf{v}^h) + \left(\mathbf{f}, \tau_\text{M} (\mathbf{a}\cdot\nabla\mathbf{v}^h + \nabla \tilde{q})\right) \label{eq:cons2}
\end{equation}
for all $(\mathbf{v}^h,q')\in\mathring{\mathcal{V}}^h\times\mathcal{Q}'$.  The final result follows by subtracting \eqref{eq:cons2} from \eqref{eq:cons1}.
\end{proof}

It should be noted that the above lemma does not hold if Assumption \ref{ass3} is not satisfied.  This is due to the fact that $\mathring{\mathcal{V}}^h \not\subset \mathring{\mathcal{V}}$ for a velocity-pressure pair that is not divergence-conforming.

To continue, we introduce the norm
\begin{align}
     \nonumber\vertiii{(\mathbf{v},q)}^2 =&~k(\mathbf{v},\mathbf{v})
     +\left(\mathbf{a}\cdot\nabla\mathbf{v} + \nabla q,\tau_\text{M}(\mathbf{a}\cdot\nabla\mathbf{v}+\nabla q)\right)\text{ .}
\end{align}
The following coercivity result then holds.

\begin{lemma}[Coercivity]\label{coercivity}
Provided that Assumptions \ref{ass1} and \ref{ass2} are satisfied, it holds that
\begin{equation}
    A_\text{red}\left((\mathbf{v}^h,\tilde{q}),(\mathbf{v}^h,\tilde{q})\right) \geq \frac{1}{2} \vertiii{(\mathbf{v}^h,\tilde{q})}^2
\end{equation}
for all $(\mathbf{v}^h,\tilde{q})\in \mathring{\mathcal{V}}^h\times\tilde{\mathcal{Q}}$.
\end{lemma}

\begin{proof}
Let $(\mathbf{v}^h,\tilde{q})\in \mathring{\mathcal{V}}^h\times\tilde{\mathcal{Q}}$.  Expand
\begin{equation}
A_\text{red}\left((\mathbf{v}^h,\tilde{q}),(\mathbf{v}^h,\tilde{q})\right) = c(\mathbf{a},\mathbf{v}^h,\mathbf{v}^h) + k(\mathbf{v}^h,\mathbf{v}^h)
    +\left(\mathbf{a}\cdot\nabla\mathbf{v}^h - \nabla\cdot(2\nu\nabla^s\mathbf{v}^h) + 
    \nabla \tilde{q},\tau_\text{M}(\mathbf{a}\cdot\nabla\mathbf{v}^h + \nabla \tilde{q})\right)\text{ .}
\end{equation}
It is easily shown that $c(\mathbf{a},\mathbf{v}^h,\mathbf{v}^h) = 0$, so it follows that
\begin{align}
A_\text{red}\left((\mathbf{v}^h,\tilde{q}),(\mathbf{v}^h,\tilde{q})\right) = \vertiii{(\mathbf{v}^h,\tilde{q})}^2
    -\left(\nabla\cdot(2\nu\nabla^s\mathbf{v}^h),\tau_\text{M}(\mathbf{a}\cdot\nabla\mathbf{v}^h + \nabla \tilde{q})\right) \text{ .}
\end{align}
Using Young's inequality,
\begin{align}
    \left(\nabla\cdot(2\nu\nabla^s\mathbf{v}^h),\tau_\text{M}(\mathbf{a}\cdot\nabla\mathbf{v}^h + \nabla \tilde{q})\right) \leq{}& \frac{1}{2}\left(\nabla\cdot(2\nu\nabla^s\mathbf{v}^h),\tau_\text{M}\nabla\cdot(2\nu\nabla^s\mathbf{v}^h)\right) + \nonumber \\
    & \frac{1}{2}\left(\mathbf{a}\cdot\nabla\mathbf{v}^h + \nabla \tilde{q},\tau_\text{M}(\mathbf{a}\cdot\nabla\mathbf{v}^h + \nabla \tilde{q})\right)\text{ .}
\end{align}
Since Assumptions \ref{ass1} and \ref{ass2} are satisfied,
\begin{align}
    \left(\nabla\cdot(2\nu\nabla^s\mathbf{v}^h),\tau_\text{M}\nabla\cdot(2\nu\nabla^s\mathbf{v}^h)\right)\leq k(\mathbf{v}^h,\mathbf{v}^h)\text{ ,}
\end{align}
thus
\begin{equation}
    \left(\nabla\cdot(2\nu\nabla^s\mathbf{v}^h),\tau_\text{M}(\mathbf{a}\cdot\nabla\mathbf{v}^h + \nabla \tilde{q})\right) \leq \frac{1}{2} \vertiii{(\mathbf{v}^h,\tilde{q})}^2\text{ .}
\end{equation}
The desired result immediately follows.
\end{proof}

Using the additional norm
\begin{align}
     \nonumber\vertiii{(\mathbf{v},q)}_+^2 =&~\nonumber\vertiii{(\mathbf{v},q)}^2
     +\left(\nabla\cdot(2\nu\nabla^s\mathbf{v}),\tau_\text{M}\nabla\cdot(2\nu\nabla^s\mathbf{v})\right)
     +\left(\mathbf{v},\tau_\text{M}^{-1}\mathbf{v}\right)\text{ ,}
\end{align}
we can show the following continuity result.

\begin{lemma}[Continuity]\label{continuity}
It holds that
\begin{equation}
    A_\text{red}\left((\mathbf{w},r),(\mathbf{v}^h,\tilde{q})\right) \leq 3 \vertiii{(\mathbf{w},r)}_+ \vertiii{(\mathbf{v}^h,\tilde{q})} 
\end{equation}
for all $(\mathbf{w},r)\in \left(\mathring{\mathcal{V}} \cap H^2(\Omega))^d \right) \times \left( \mathcal{Q} \cap H^1(\Omega) \right)$ and $(\mathbf{v}^h,\tilde{q})\in \mathring{\mathcal{V}}^h\times\tilde{\mathcal{Q}}$.
\end{lemma}

\begin{proof}
Let $(\mathbf{w},r)\in \left(\mathring{\mathcal{V}} \cap H^2(\Omega))^d \right) \times \left( \mathcal{Q} \cap H^1(\Omega) \right)$ and  $(\mathbf{v}^h,\tilde{q})\in \mathring{\mathcal{V}}^h\times\tilde{\mathcal{Q}}$.  Expand
\begin{align}
    A_\text{red}\left((\mathbf{w},r),(\mathbf{v}^h,\tilde{q})\right) ={}& c(\mathbf{a},\mathbf{w},\mathbf{v}^h) + k(\mathbf{w},\mathbf{v}^h)
    +\left(\mathbf{a}\cdot\nabla\mathbf{w} - \nabla\cdot(2\nu\nabla^s\mathbf{w}) + 
    \nabla r,\tau_\text{M}(\mathbf{a}\cdot\nabla\mathbf{v}^h + \nabla \tilde{q})\right) \nonumber \\
    ={}& (\mathbf{a}\cdot\nabla\mathbf{w},\mathbf{v}^h) + (2\nu\nabla^s\mathbf{w},\nabla^s\mathbf{v}^h)\nonumber \\ 
    &+ \left(\mathbf{a}\cdot\nabla\mathbf{w} + 
    \nabla r,\tau_\text{M}(\mathbf{a}\cdot\nabla\mathbf{v}^h + \nabla \tilde{q})\right) - \left(\nabla\cdot(2\nu\nabla^s\mathbf{w}),\tau_\text{M}(\mathbf{a}\cdot\nabla\mathbf{v}^h + \nabla \tilde{q})\right)\text{ .}
\end{align}
Since $\nabla \cdot \mathbf{a} = 0$ and $\mathbf{w} \in \mathring{\mathcal{V}}$, it follows that
\begin{align}
    (\mathbf{a}\cdot\nabla\mathbf{w},\mathbf{v}^h) = -(\mathbf{w},\mathbf{a}\cdot\nabla\mathbf{v}^h) + (\nabla \cdot \mathbf{w},\tilde{q}) = -(\mathbf{w},\mathbf{a}\cdot\nabla\mathbf{v}^h + \nabla \tilde{q})\text{ .}
\end{align}
Consequently,
\begin{align}
    A_\text{red}\left((\mathbf{w},r),(\mathbf{v}^h,\tilde{q})\right) 
    ={}& -(\mathbf{w},\mathbf{a}\cdot\nabla\mathbf{v}^h + \nabla \tilde{q}) + (2\nu\nabla^s\mathbf{w},\nabla^s\mathbf{v}^h)\nonumber \\ 
    &+ \left(\mathbf{a}\cdot\nabla\mathbf{w} + 
    \nabla r,\tau_\text{M}(\mathbf{a}\cdot\nabla\mathbf{v}^h + \nabla \tilde{q})\right) - \left(\nabla\cdot(2\nu\nabla^s\mathbf{w}),\tau_\text{M}(\mathbf{a}\cdot\nabla\mathbf{v}^h + \nabla \tilde{q})\right)\text{ .}
    %\leq{}& 4 \vertiii{(\mathbf{w},r)}_+ \vertiii{(\mathbf{v}^h,\tilde{q})}\text{ .}
\end{align}
By the Cauchy-Schwarz inequality,
\begin{align}
(2\nu\nabla^s\mathbf{w},\nabla^s\mathbf{v}^h) + \left(\mathbf{a}\cdot\nabla\mathbf{w} + 
    \nabla r,\tau_\text{M}(\mathbf{a}\cdot\nabla\mathbf{v}^h + \nabla \tilde{q})\right) \leq \vertiii{(\mathbf{w},r)} \vertiii{(\mathbf{v}^h,\tilde{q})} \leq \vertiii{(\mathbf{w},r)}_+ \vertiii{(\mathbf{v}^h,\tilde{q})}\text{ .}
\end{align}
Moreover,
\begin{align}
(\mathbf{w},\mathbf{a}\cdot\nabla\mathbf{v}^h + \nabla \tilde{q}) \leq \vertiii{(\mathbf{w},r)}_+ \vertiii{(\mathbf{v}^h,\tilde{q})}
\end{align}
and
\begin{align}
\left(\nabla\cdot(2\nu\nabla^s\mathbf{w}),\tau_\text{M}(\mathbf{a}\cdot\nabla\mathbf{v}^h + \nabla \tilde{q})\right) \leq \vertiii{(\mathbf{w},r)}_+ \vertiii{(\mathbf{v}^h,\tilde{q})}\text{ .}
\end{align}
The desired result immediately follows.
\end{proof}

Finally, using Lemmas \ref{consistency}-\ref{continuity}, we can show the following result.

\begin{theorem}[Error Estimate]\label{estimate}
Provided that Assumptions \ref{ass1}-\ref{ass3} are satisfied, it holds that
\begin{equation}
    \vertiii{(\mathbf{u} - \mathbf{u}^h,p - \tilde{p})} \leq 7 \inf_{\left( \mathbf{w}^h, \tilde{r} \right) \in \mathring{\mathcal{V}}^h \times \tilde{\mathcal{Q}}} \vertiii{(\mathbf{u} - \mathbf{w}^h,p - \tilde{r})}_+
\end{equation}
if the exact solution $(\mathbf{u},p)\in \mathcal{V} \times \mathcal{Q}$ to the Oseen problem lies in $(H^2(\Omega))^d \times H^1(\Omega)$.
\end{theorem}

\begin{proof}
Let $\left( \mathbf{w}^h, \tilde{r} \right) \in \mathring{\mathcal{V}}^h \times \tilde{\mathcal{Q}}$.  By Lemma \ref{coercivity},
\begin{align}
    \frac{1}{2} \vertiii{(\mathbf{u}^h - \mathbf{w}^h,\tilde{p} - \tilde{r})}^2 \leq A_\text{red}\left((\mathbf{u}^h - \mathbf{w}^h,\tilde{p} - \tilde{r}),(\mathbf{u}^h - \mathbf{w}^h,\tilde{p} - \tilde{r})\right)\text{ .}
\end{align}
By Lemma \ref{consistency},
\begin{align}
    \frac{1}{2}\vertiii{(\mathbf{u}^h - \mathbf{w}^h,\tilde{p} - \tilde{r})}^2 \leq A_\text{red}\left((\mathbf{u} - \mathbf{w}^h,p - \tilde{r}),(\mathbf{u}^h - \mathbf{w}^h,\tilde{p} - \tilde{r})\right)\text{ .}
\end{align}
By Lemma \ref{continuity},
\begin{align}
    \frac{1}{2}\vertiii{(\mathbf{u}^h - \mathbf{w}^h,\tilde{p} - \tilde{r})}^2 \leq 3 \vertiii{(\mathbf{u} - \mathbf{w}^h,p - \tilde{r})}_+ \vertiii{(\mathbf{u}^h - \mathbf{w}^h,\tilde{p} - \tilde{r})} \text{ ,}
\end{align}
so
\begin{align}
\vertiii{(\mathbf{u}^h - \mathbf{w}^h,\tilde{p} - \tilde{r})} \leq 6 \vertiii{(\mathbf{u} - \mathbf{w}^h,p - \tilde{r})}_+\text{ .}
\end{align}
By the triangle inequality,
\begin{align}
\vertiii{(\mathbf{u} - \mathbf{u}^h,p - \tilde{p})} \leq \vertiii{(\mathbf{u}^h - \mathbf{w}^h,\tilde{p} - \tilde{r})} + \vertiii{(\mathbf{u} - \mathbf{w}^h,p - \tilde{r})} \leq 7 \vertiii{(\mathbf{u} - \mathbf{w}^h,p - \tilde{r})}_+\text{ .}
\end{align}
The desired result follows since $\left( \mathbf{w}^h, \tilde{r} \right) \in \mathring{\mathcal{V}}^h \times \tilde{\mathcal{Q}}$ is arbitrary.
\end{proof}

The above result indicates that our VMS-based formulation is quasi-optimal with respect to an SUPG/PSPG-like norm for divergence-conforming discretizations of the Oseen problem.  Note that the above result also indicates that while the velocity error depends on the approximation properties of the fine-scale pressure space, it does not depend on the approximation properties of the coarse-scale pressure space.  Thus any spurious influence of the pressure field on the coarse-scale velocity field may be mitigated by refining just the fine-scale pressure space rather than both the coarse-scale and fine-scale pressure spaces.  This is an important observation as one can refine the fine-scale pressure space without upsetting inf-sup stability while the same cannot be said for the coarse-scale pressure space.

In order to extract convergence rates for our VMS-based formulation, we must make additional assumptions regarding the approximation power of the velocity and pressure spaces $\mathcal{V}^h$ and $\tilde{\mathcal{Q}}$ and the specific form of the stabilization parameter $\tau_\text{M}$.

\begin{assumption}\label{ass4}
There exist integers $k_v$ and $k_q$, interpolation operators $\mathcal{I}_v : \mathring{\mathcal{V}} \rightarrow \mathring{\mathcal{V}}^h$ and $\mathcal{I}_q : \mathcal{Q} \rightarrow \tilde{\mathcal{Q}}$, and an interpolation constant $C_\text{interp}$ independent of the global mesh size $h = \max_e h_e$ such that, for every $\mathbf{v} \in \mathring{\mathcal{V}} \cap (H^{k_v+1}(\Omega))^d$ and integer $0 \leq l \leq k_v + 1$,
\begin{equation}
\sum_{e=1}^{n_{el}} | \mathbf{v} - \mathcal{I}_v \mathbf{v} |^2_{(H^l(\Omega^e))^d} \leq C_\text{interp} h^{2(k_v+1-l)} | \mathbf{v} |^2_{(H^{k_v+1}(\Omega))^d},
\end{equation}
and, for every $q \in \mathcal{Q} \cap H^{k_q+1}(\Omega)$ and integer $0 \leq l \leq k_q + 1$,
\begin{equation}
\sum_{e=1}^{n_{el}} | q - \mathcal{I}_q q |^2_{H^l(\Omega^e)} \leq C_\text{interp} h^{2(k_q+1-l)} | q |^2_{H^{k_q+1}(\Omega)}.
\end{equation}
We refer to $k_v$ and $k_q$ as the degrees of the coarse-scale velocity and fine-scale pressure spaces, respectively.
\end{assumption}

\begin{assumption}\label{ass5}
For each element $\Omega^e \in \mathcal{M}$,
\begin{equation}
     \tau_\text{M} = \frac{h_e}{2 |\mathbf{a}|} f(\gamma_e),
\end{equation}
where
\begin{equation}
    \gamma_e =  \frac{| \mathbf{a} | h_e}{2\nu}
\end{equation}
is the element P\'{e}clet number and $f$ is a monotone function satisfying
\begin{equation}
    f(\gamma_e) \leq \min\left\{ 1, \frac{4 \gamma_e}{C_\text{inv}} \right\},
\end{equation}
where $C_\text{inv}$ is the constant associated with Assumption \ref{ass1}.
\end{assumption}

Assumption \ref{ass5} holds for standard choices for the stabilization parameter \cite{Franca1992}, and Assumption \ref{ass2} is automatically satisfied if Assumption \ref{ass5} is satisfied.  Assumption \ref{ass4}, however, is a non-standard assumption.  Satisfaction of Assumption \ref{ass4} requires the existence of a divergence-conforming projection operator for the coarse-scale velocity space.  Such projection operators exist, for instance, for divergence-conforming B-spline discretizations \cite{Buffa2011} as well as a selection of divergence-conforming finite element discretizations \cite{scott1985norm,guzman2014conforming}.  It should be noted, however, that divergence-conforming finite element discretizations typically employ discontinuous pressure approximation.
Provided a divergence-conforming projection operator does exist, Assumption \ref{ass4} typically holds for quasi-uniform meshes $\mathcal{M}$.  Equipped with the above assumptions, we have the following result.

\begin{theorem}[Rates of Convergence]\label{convergence}
Suppose that the advection velocity $\mathbf{a}$ and the kinematic viscosity $\nu$ are constant, and suppose that the mesh $\mathcal{M}$ is quasi-uniform.  Provided that Assumptions \ref{ass1}-\ref{ass5} are satisfied, it holds that
\begin{equation}
    \vertiii{(\mathbf{u} - \mathbf{u}^h,p - \tilde{p})}^2 \leq C_\text{bound} \left( \max\{ | \mathbf{a} | h, \nu \} h^{2k_v} | \mathbf{u} |^2_{(H^{k_v+1}(\Omega))^d} + \max\{ | \mathbf{a} | h, \nu \}^{-1} h^{2k_q+2} | p |^2_{H^{k_q+1}(\Omega)} \right)\text{ ,}
\end{equation}
where $C_\text{bound}$ is a dimensionless constant independent of the global mesh size $h = \max_e h_e$, if the exact solution $(\mathbf{u},p)\in \mathcal{V} \times \mathcal{Q}$ to the Oseen problem lies in $(H^{k_v+1}(\Omega))^d \times H^{k_q+1}(\Omega)$.
\end{theorem}

\begin{proof}
By Theorem \ref{estimate}, we have
\begin{equation}
    \vertiii{(\mathbf{u} - \mathbf{u}^h,p - \tilde{p})}^2 \leq 49 \vertiii{(\mathbf{u} - \mathcal{I}_v \mathbf{u},p - \mathcal{I}_q p)}^2_+\text{ .} \label{eq:rates_first}
\end{equation}
However,
\begin{align}
    \vertiii{(\mathbf{u} - \mathcal{I}_v \mathbf{u},p - \mathcal{I}_q p)}^2_+ &= k(\mathbf{\mathbf{u} - \mathcal{I}_v \mathbf{u}},\mathbf{u} - \mathcal{I}_v \mathbf{u}) \nonumber \\
     &\phantom{=} +\left(\mathbf{a}\cdot\nabla \left(\mathbf{u} - \mathcal{I}_v \mathbf{u}\right) + \nabla \left(p - \mathcal{I}_q p\right), \tau_\text{M} (\mathbf{a}\cdot\nabla\left(\mathbf{u} - \mathcal{I}_v \mathbf{u}\right)+\nabla \left(p - \mathcal{I}_q p\right))\right) \nonumber
     \\&\phantom{=} + \left(\nabla\cdot(2\nu\nabla^s\left(\mathbf{u} - \mathcal{I}_v \mathbf{u}\right)),\tau_\text{M}\nabla\cdot(2\nu\nabla^s\left(\mathbf{u} - \mathcal{I}_v \mathbf{u}\right))\right)
     \nonumber \\
     &\phantom{=}+\left(\mathbf{u} - \mathcal{I}_v \mathbf{u},\tau_\text{M}^{-1}\left(\mathbf{u} - \mathcal{I}_v \mathbf{u}\right)\right)\text{ .}
\end{align}
Assumption \ref{ass4} immediately gives
\begin{equation}
    k(\mathbf{\mathbf{u} - \mathcal{I}_v \mathbf{u}},\mathbf{u} - \mathcal{I}_v \mathbf{u}) \leq C_\text{interp} \nu h^{2k_v} | \mathbf{u} |^2_{(H^{k_v+1}(\Omega))^d}\text{ .}
\end{equation}
Assumptions \ref{ass4} and \ref{ass5} give
\begin{equation}
    \left(\nabla\cdot(2\nu\nabla^s\left(\mathbf{u} - \mathcal{I}_v \mathbf{u}\right)),\tau_\text{M}\nabla\cdot(2\nu\nabla^s\left(\mathbf{u} - \mathcal{I}_v \mathbf{u}\right))\right) \leq \frac{C_\text{interp}}{C_\text{inv}} \nu h^{2k_v} | \mathbf{u} |^2_{(H^{k_v+1}(\Omega))^d}
\end{equation}
and
\begin{equation}
    \left(\mathbf{u} - \mathcal{I}_v \mathbf{u},\tau_\text{M}^{-1}\left(\mathbf{u} - \mathcal{I}_v \mathbf{u}\right)\right) \leq C_\text{interp} C_\text{mesh} \max\{ 2, C_\text{inv} \} \max\{ | \mathbf{a} | h, \nu \} h^{2k_v} | \mathbf{u} |^2_{(H^{k_v+1}(\Omega))^d}\text{ ,}
\end{equation}
where $C_\text{mesh} > 0$ is a dimensionless constant that depends on the ratio $\max_e h_e/\min_e h_e$ and the monotone function $f$ appearing in Assumption \ref{ass5}. The Cauchy-Schwarz and Young's inequalities give
\begin{align}
    \left(\mathbf{a}\cdot\nabla \left(\mathbf{u} - \mathcal{I}_v \mathbf{u}\right) + \nabla \left(p - \mathcal{I}_q p\right), \tau_\text{M} (\mathbf{a}\cdot\nabla\left(\mathbf{u} - \mathcal{I}_v \mathbf{u}\right)+\nabla \left(p - \mathcal{I}_q p\right))\right) \leq \phantom{ ,} \nonumber \\ 2 \left( \left(\mathbf{a}\cdot\nabla \left(\mathbf{u} - \mathcal{I}_v \mathbf{u}\right), \tau_\text{M} \mathbf{a}\cdot\nabla \left(\mathbf{u} - \mathcal{I}_v \mathbf{u}\right) \right) + \left( \nabla \left(p - \mathcal{I}_q p\right), \tau_\text{M} \nabla \left(p - \mathcal{I}_q p\right) \right) \right) \text{ ,}
\end{align}
and Assumptions \ref{ass4} and \ref{ass5} give
\begin{equation}
    \left(\mathbf{a}\cdot\nabla \left(\mathbf{u} - \mathcal{I}_v \mathbf{u}\right), \tau_\text{M} \mathbf{a}\cdot\nabla \left(\mathbf{u} - \mathcal{I}_v \mathbf{u}\right) \right) \leq \frac{1}{2} C_\text{interp} | \mathbf{a} | h^{2k_v+1} | \mathbf{u} |^2_{(H^{k_v+1}(\Omega))^d}
\end{equation}
and
\begin{equation}
    \left( \nabla \left(p - \mathcal{I}_q p\right), \tau_\text{M} \nabla \left(p - \mathcal{I}_q p\right) \right) \leq C_\text{interp} \max\left\{ \frac{1}{2}, \frac{1}{C_\text{inv}} \right\} \max\left\{ | \mathbf{a} | h, \nu \right\}^{-1} h^{2k_q+2} | p |^2_{H^{k_q+1}(\Omega)} \text{ .} \label{eq:rates_last}
\end{equation}
The desired result follows from combining \eqref{eq:rates_first}-\eqref{eq:rates_last}.
\end{proof}

In the advection-dominated limit, Theorem \ref{convergence} yields
\begin{equation}
\vertiii{(\mathbf{u} - \mathbf{u}^h,p - \tilde{p})}^2 \sim |\mathbf{a}| h^{2k_v+1} | \mathbf{u} |^2_{(H^{k_v+1}(\Omega))^d} + |\mathbf{a}|^{-1} h^{2k_q+1} | p |^2_{H^{k_q+1}(\Omega)} \text{ .}
\end{equation}
The above equation suggests that one should select the coarse-scale velocity and fine-scale pressure spaces so that $k_v = k_q$ in order to balance the velocity and pressure contributions to the error.  On the other hand, in the diffusion-dominated limit,
Theorem \ref{convergence} yields
\begin{equation}
\vertiii{(\mathbf{u} - \mathbf{u}^h,p - \tilde{p})}^2 \sim \nu h^{2k_v} | \mathbf{u} |^2_{(H^{k_v+1}(\Omega))^d} + \nu^{-1} h^{2k_q+2} | p |^2_{H^{k_q+1}(\Omega)} \text{ .}
\end{equation}
The above equation suggests that one should select the coarse-scale velocity and fine-scale pressure spaces so that $k_v = k_q + 1$ in order to balance the velocity and pressure contributions to the error.  Since the advection-dominated limit is of more practical interest than the diffusion-dominated limit, it is suggested to select the coarse-scale velocity and fine-scale pressure spaces so that $k_v = k_q$.

Theorems \ref{estimate} and \ref{convergence} suggest that, much like the classical SUPG/PSPG method \cite{Franca1992}, our VMS-based formulation is robust with respect to both advection and diffusion.  However, Theorems \ref{estimate} and \ref{convergence} only apply to divergence-conforming discretizations.  To prove analogous results for non-divergence-conforming discretizations requires the use of an inf-sup stability analysis rather than a coercivity stability analysis.  Such a study is relegated to future work.

\section{Numerical results}\label{sec:num-examples}

We finish this paper by applying our VMS-based formulation to the numerical solution of the incompressible Navier-Stokes problem.  We first examine the convergence properties of our formulation using steady and unsteady exact solutions before exploring the applicability of our VMS model as an implicit LES filter using the three-dimensional Taylor-Green vortex problem.  For all of our numerical studies, we employ divergence-conforming B-splines for spatial discretization \cite{Evans2013,Evans:2013fe}.  Isogeometric divergence-conforming discretizations yield pointwise divergence-free velocity fields, so the convergence analysis of Section \ref{sec:oseen-conv} directly applies to such discretizations.

\subsection{Implementation using FEniCS and tIGAr}\label{sec:implementation}

To implement the VMS-based formulation described in Section \ref{sec:formulation} in a transparent way, we employ finite element automation software from the FEniCS Project \cite{Logg2012}, which allows symbolic specification of variational problems in Unified Form Language (UFL) \cite{Alnaes2014}.  These specifications are then compiled \cite{Kirby2006} into efficient routines suitable for high-performance simulations using the solver DOLFIN \cite{Logg:2010}.  FEniCS can be used for isogeometric analysis (IGA) via the library tIGAr \cite{Kamensky2019}.  Some of the complicated variational forms generated by tIGAr benefit from the use of experimental support in FEniCS for and advanced form compiler called TSFC \cite{Homolya2018}.  

We solve the saddle-point problem implied by Problem $(V^h)$ using an iterative strategy based on the iterated penalty method advocated by Morgan and Scott \cite{Morgan2018}.  This is possible since the spatial discretization we consider is divergence-conforming and hence the exact mass conservation constraint is recovered with our formulation.  We extend the iterated penalty method, as in \cite{Kamensky2019}, to perform penalty iterations in tandem with nonlinear iterations.  It is also possible to use the multigrid method proposed in \cite{coley2018geometric} to solve the saddle point problem.

%In the case of non-conforming finite element discretizations, the penalty term is replaced with a perturbation to the coarse-scale constraint Lagrangian \cite{Simo1985}.  

The use of FEniCS leads to concise and comprehensible implementations.  Readers interested in the precise definitions of stabilization parameters, solver configurations, computational performance, results for additional cases, etc. are encouraged to examine, run, and/or modify the open-source code examples accompanying this paper \cite{examples-repo}.  

%[DK] We need to decide on some details regarding the examples repository before flipping it to public.  First, should we leave in the Taylor--Hood codes, and just keep using the same repository for the follow-up paper?  Second, should we re-name it to "discretely-div-free" instead of "weakly-div-free"?  (See: https://help.github.com/en/github/administering-a-repository/renaming-a-repository)  I think the implementation alone of Taylor--Hood is obvious, so we're not really "giving away" anything before the second paper is out; the main novelty will be the analysis.

\subsection{Convergence tests}\label{sec:conv-test}
We first demonstrate numerically that the convergence analysis of Section \ref{sec:oseen-conv} for the Oseen problem can be extrapolated to the steady and unsteady incompressible Navier--Stokes problems.  

\subsubsection{Steady flow: The regularized lid-driven cavity}
We begin our numerical tests with the simplest extension of the Oseen problem, namely, steady Navier--Stokes flow.  We test convergence with a variant of the manufactured solution called the regularized lid-driven cavity, as proposed in \cite{Shih1989} and studied also by \cite{VanOpstal2017}.  In particular, we choose, {\em a priori}, the exact velocity solution from \cite{Shih1989}, 
\begin{align}
    \mathbf{u}(\mathbf{x}) =& \left(x_1^4 - 2x_1^3 + x_1^2\right)\left(4x_2^3 - 2x_2\right)\mathbf{e}_1 - 8\left(4x_1^3 - 6x_1^2 + 2x_1\right)\left(x_2^4 - x_2^2\right)\mathbf{e}_2\text{ ,}
\end{align}
and the pressure\footnote{This differs from \cite{Shih1989}, whose expression for $p$ emerges from complicated manual manipulations.  When manufacturing solutions automatically, one can choose an arbitrary smooth scalar field.}
\begin{equation}
    p(\mathbf{x}) = \sin(\pi x_1)\sin(\pi x_2)\text{ .}
\end{equation}
The corresponding source term is generated automatically from the strong form of the problem, using UFL.  The velocity solution resembles the classic lid-driven cavity benchmark but without the corner singularities.  We apply the exact velocity solution as a Dirichlet boundary condition on the unit square domain.  We take unit length and velocity scales, and we present results for a Reynolds number of $Re=100$.

%\begin{figure}[t!]\centering
%\includegraphics[width=0.5\textwidth]{reg-ldc-snapshot.png}
%\caption{Velocity field for the regularized lid-driven cavity, colored by magnitude, with streamlines traced in black.}
%\label{fig:reg-ldc-snapshot}
%\end{figure}

\begin{figure}[tp!]\centering
\includegraphics[width=0.85\textwidth]{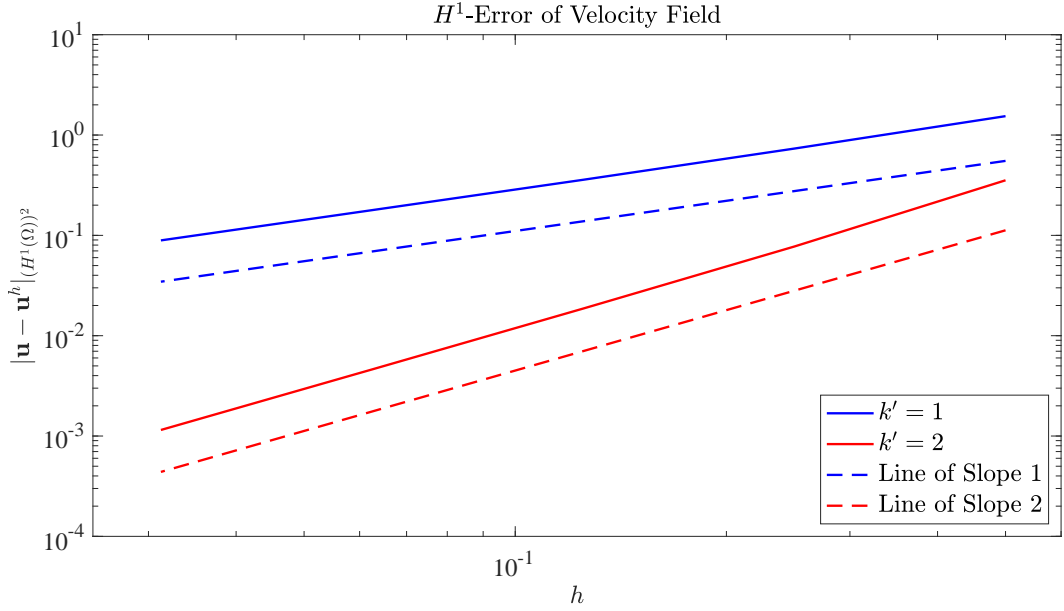}
\caption{Convergence of the $H^1$-semi-norm of the velocity error for the regularized lid-driven cavity using divergence-conforming B-splines of degree $k'=1$ and $k'=2$.}
\label{fig:reg-ldc-dcb-err}
\end{figure}

We test our VMS-based formulation for this problem using divergence-conforming B-splines, selecting the fine-scale pressure space to be equal to the coarse-scale pressure space.  This selection results in the coarse-scale velocity and fine-scale pressure spaces being complete up to the same polynomial degree. The stabilization parameter is selected as
\begin{equation}
     \tau_\text{M} = \left( \mathbf{u}^h \cdot \mathbf{G} \cdot \mathbf{u}^h + C^2_\text{inv} \nu^2 \mathbf{G} : \mathbf{G} \right)^{-1/2} \text{ ,} \label{eq:tau_dyn}
\end{equation}
where $C_\text{inv}$ is the constant associated with Assumption \ref{ass1}, $\mathbf{G} = \frac{\partial \boldsymbol{\xi}}{\partial \mathbf{x}}^T \frac{\partial \boldsymbol{\xi}}{\partial \mathbf{x}}$, and $\frac{\partial \boldsymbol{\xi}}{\partial \mathbf{x}}$ is the inverse Jacobian of the map between the parent element and the physical element.  See \cite{Bazilevs07b} for more details.  Figure \ref{fig:reg-ldc-dcb-err} shows the $H^1$-semi-norm of the velocity error for splines of degree $k'=1$ and $k'=2$.  The notation ``$k'$\,'' is understood in the sense defined by \cite{Evans2011}, i.e., the degree up to which the coarse-scale velocity approximation space is polynomially complete.  Note that optimal convergence rates are observed for both $k'=1$ and $k'=2$.

%We first test the full formulation using the non-div-conforming triangular Taylor--Hood elements (i.e., quadratic Lagrange elements for velocity components and linear Lagrange elements for pressure).  The $H^1$ seminorm of the velocity field is seen, in Figure \ref{fig:reg-ldc-th-err}, to converge at the optimal rate with respect to $h$-refinement.  
%
%[DK] TODO: comment on expected sub-optimality, once there is a number equation to reference.
%
%Next, we consider the case of div-conforming B-splines, for which treatment of convection is simplified.  Figure \ref{fig:reg-ldc-dcb-err} shows optimal convergence rates for splines of degree $k'=1$ and $k'=2$.  The notation ``$k'$\,'' is understood in the sense defined by \cite{Evans2011}, i.e., the degree up to which the velocity approximation space is polynomially-complete.  
%
%[DK] This was actually a point of confusion for one reviewer of my tIGAr paper.  They asked: Why are you using $k'=1$?  Isn't that just bilinear quads?  Where is the IGA?
%(N.b. that some univariate splines entering the tensor products used to define velocity components have degree $k'+1$.)

%\begin{figure}[!htbp]\centering
%\includegraphics[width=0.7\textwidth]{reg-ldc-th-err.pdf}
%\caption{Convergence of $H^1$ velocity error for the regularized lid-driven cavity, using Taylor--Hood elements.}
%\label{fig:reg-ldc-th-err}
%\end{figure}

\begin{remark}\label{rem:bcs}
For our divergence-conforming B-spline computations, we differ from the formulation of \cite{Evans2013} by applying Dirichlet boundary conditions strongly in the tangential direction, via a divergence-free lifting of the boundary data (computed, in this case, as an $L^2$ projection of the exact solution onto the solenoidal subspace of $\mathcal{V}^h$).  This is most consistent with the formulation and convergence analysis in Sections \ref{sec:formulation} and \ref{sec:oseen-conv}, where strong enforcement has been assumed for simplicity.  However, in practice, we recommend the weak enforcement outlined by \cite{Evans2013}, which requires much less boundary-layer resolution to obtain accurate solutions to realistic flow problems.
\end{remark}

\subsubsection{Unsteady flow: Two-dimensional Taylor--Green vortex flow}
We continue our numerical tests with unsteady Navier-Stokes flow.  We test convergence in the unsteady setting using two-dimensional Taylor-Green vortex flow, whose full problem specification and exact solution are given in \cite[Section 9.10.1]{Evans2011}.  Two-dimensional Taylor-Green vortex flow describes the behavior of a two-dimensional vortex who decays away in time due to viscous forces.  We consider both time discretizations of the full semidiscrete unsteady formulation given in Subsection \ref{sec:semi-discrete} as well as its quasi-static-subscale variant described in Subsection \ref{sec:alt-form}.  We discretize in time using the second-order-accurate implicit midpoint rule, maintain a linear proportionality between the time step size and the mesh size, and employ a sufficiently small time step size so that spatial error dominates temporal error.

\begin{figure}[!tp]\centering
\includegraphics[width=0.85\textwidth]{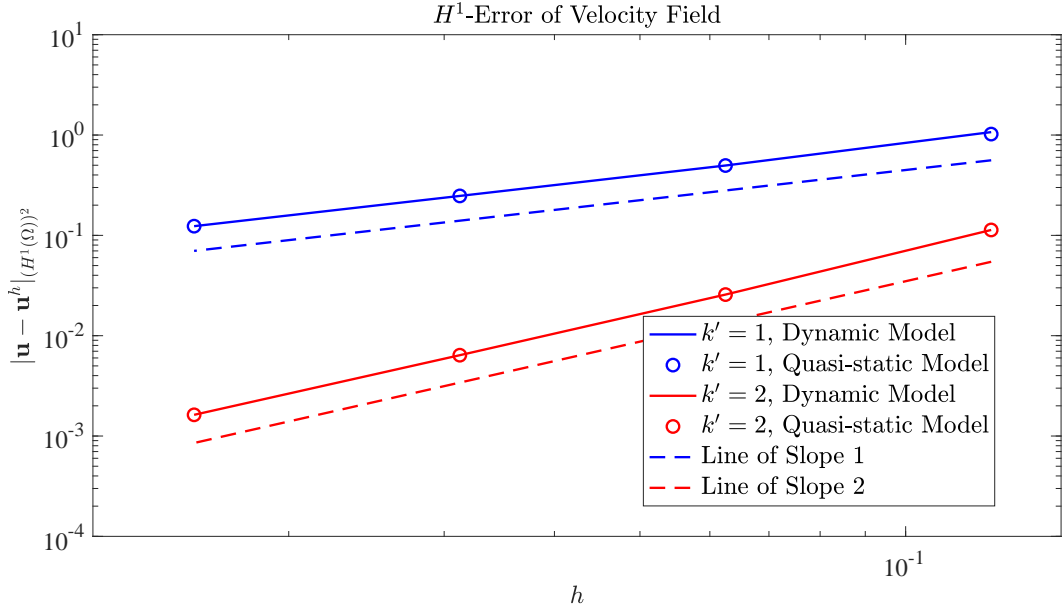}
\caption{Convergence of the $H^1$-semi-norm of the velocity error at time $T = 1$ for two-dimensional Taylor-Green vortex flow using divergence-conforming B-splines of degree $k'=1$ and $k'=2$ with the dynamic and quasi-static subgrid models.}
\label{fig:taylor-green-2d-err}
\end{figure}

We again test our VMS-based formulation using divergence-conforming B-splines, selecting the fine-scale pressure space to be equal to the coarse-scale pressure space.  For the dynamic subscale model,
the stabilization parameter is selected according to \eqref{eq:tau_dyn}.  For the quasi-static subscale model, the stabilization parameter is selected as
\begin{equation}
     \tau_\text{M} = \left( \frac{4}{\Delta t^2} + \mathbf{u}^h \cdot \mathbf{G} \cdot \mathbf{u}^h + C^2_\text{inv} \nu^2 \mathbf{G} : \mathbf{G} \right)^{-1/2} \text{ ,} \label{eq:tau_qs}
\end{equation}
where $C_\text{inv}$ is the constant associated with Assumption \ref{ass1}, $\mathbf{G} = \frac{\partial \boldsymbol{\xi}}{\partial \mathbf{x}}^T \frac{\partial \boldsymbol{\xi}}{\partial \mathbf{x}}$, $\frac{\partial \boldsymbol{\xi}}{\partial \mathbf{x}}$ is the inverse Jacobian of the map between the parent element and the physical element, and $\Delta t$ is the time step size.
This is consistent with other VMS-based methods based on quasi-static models \cite{Bazilevs07b}.  Figure \ref{fig:taylor-green-2d-err} shows the $H^1$-semi-norm of the velocity error at time $T=1$ for splines of degree $k'=1$ and $k'=2$.  Note that, as was the case for the steady manufactured solution, optimal convergence rates are observed for both $k'=1$ and $k'=2$.  Note additionally that the dynamic and quasi-static subscale models yield virtually the same error for both $k' = 1$ and $k' = 2$.

\subsection{Application to turbulent fluid flow}
We finally examine the effectiveness of our VMS-based formulation as a residual-based LES methodology for turbulent fluid flow.  In particular, we apply our formulation to the numerical simulation of three-dimensional Taylor-Green vortex flow at a Reynolds number of $Re = 1600$.  Three-dimensional Taylor-Green vortex flow is one of the simplest systems with which one can study the generation of small scales through vortex stretching and the dissipation of energy from the resulting turbulence.  The initial condition of three-dimensional Taylor-Green vortex flow is laminar and consists of purely two-dimensional streamlines, for all time $t > 0$, the flow is three-dimensional.  As the solution is evolved in time, vortex stretching generates small-scale motion and eventually causes transition to turbulence.  The initial condition for this flow is
\begin{equation}\label{eq:TGV}
{\bf u}_0\left(x,y,z\right) = \left[ \begin{array}{c}
\sin\left(x\right)\cos\left(y\right)\cos\left(z\right) \\
-\sin\left(x\right)\sin\left(y\right)\cos\left(z\right) \\
0 \end{array} \right].
\end{equation}
The flow is periodic in all three spatial directions in the domain $\Omega = \left(0,2\pi\right)^3$ and, due to inherent symmetries in the flow, the flow can be modeled within a computational domain of $\Omega^h = \left(0,\pi\right)^3$ with no-penetration and free-slip boundaries.  At low Reynolds numbers ($Re < 400$), the flow is anisotropic for all time, but as the Reynolds number increases, the flow experiences increased isotropy \cite{Brachet}.

As with the previous two numerical tests, we test our VMS-based formulation using divergence-conforming B-splines, selecting the fine-scale pressure space to be equal to the coarse-scale pressure space.  We employ the quasi-static subscale model, as quasi-static and dynamic subscale models typically return comparable results for VMS-based formulations provided the time step size is not excessively small \cite{colomes2015assessment}.  Also as before, we discretize in time using the second-order-accurate implicit midpoint rule, maintain a linear proportionality between the time step size and the mesh size, and employ a sufficiently small time step size so that spatial error dominates temporal error.  The stabilization parameter is selected according to \eqref{eq:tau_qs}.  We consider the time history of the kinetic energy dissipation rate as a quantity of interest, as is common in LES studies of three-dimensional Taylor-Green vortex flow.  The kinetic energy, $E_k$, of the flow is given by,
\begin{equation}\label{eq:KE}
E_k\left(t\right) = \frac{1}{|\Omega|} \int_\Omega \frac{{\bf u}\left(t\right) \cdot {\bf u}\left(t\right)}{2} d\Omega.
\end{equation}
The kinetic energy dissipation rate, $\epsilon$, can then be computed as
\begin{equation}\label{eq:KEDR}
\epsilon(t) = - \frac{dE_k}{dt}(t).
\end{equation}
For three-dimensional Taylor-Green vortex flow, the kinetic energy dissipation rate is identically the viscous dissipation,
\begin{equation}\label{eq:KEDR_alt}
\epsilon(t) = \int_{\Omega} 2\nu\nabla^s\mathbf{u}(t):\nabla^s\mathbf{u}(t) d\Omega.
\end{equation}
%
%where the enstrophy is computed as
%\begin{equation}\label{eq:enstrophy}
%\eta = \frac{1}{|\Omega|} \int_\Omega \frac{\bm{\omega} \cdot \bm{\omega}}{2} d\Omega,
%\end{equation}
%where $\bm{\omega} = \nabla \times {\bf u}$ is the vorticity. 
For LES, application of \eqref{eq:KEDR_alt} to the resolved velocity field results in the resolved dissipation rate.  This represents the amount of kinetic energy dissipation that is present in the coarse scales.  The difference of the total dissipation rate and the resolved dissipation rate allows for the calculation of the model dissipation rate, which gives a measure of how much dissipation a particular subgrid scale model provides.

\begin{figure}[tp!]\centering
\includegraphics[width=0.48\textwidth]{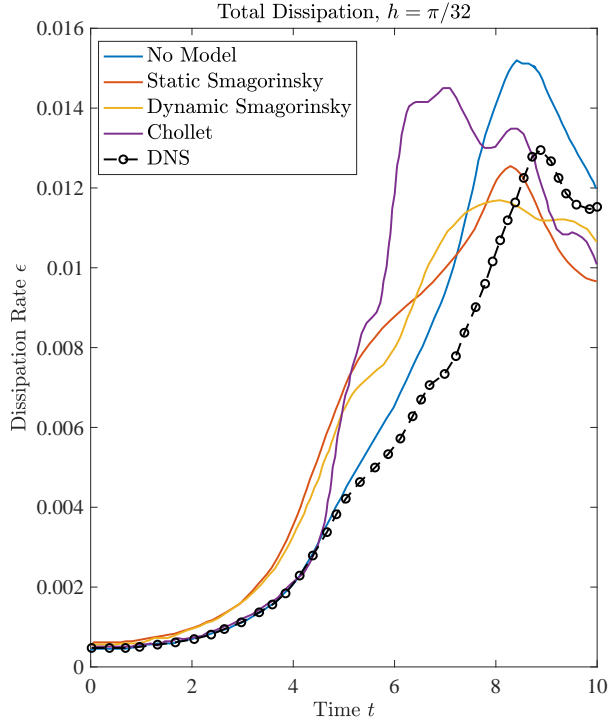}
\caption{Total dissipation rate time history for three-dimensional Taylor-Green vortex flow at $Re$ = 1600 using divergence-conforming B-splines of degree $k' = 2$ and $h = \frac{\pi}{32}$. Comparison of classical LES subgrid models (no model, Static Smagorinsky, Dynamic Smagorinsky, Chollet) with DNS \cite{Brachet}.  Results adopted from \cite{evans2018residual}.}
\label{fig:3d-tgv-classical}
\end{figure}

Before examining results obtained using our VMS-based formulation, we first examine results obtained using classical LES subgrid models \cite{evans2018residual}.  In Fig. \ref{fig:3d-tgv-classical}, we examine the total dissipation rate time history for the case of no model, the static Smagorinsky model \cite{smagorinsky1963general}, the dynamic Smagorinsky model \cite{Germano}, and the Chollet model \cite{Chollet81}, all for a divergence-conforming discretization of degree $k' = 2$ and a mesh of $32^3$ elements.  For this mesh, the flow is moderately unresolved.  This is intentional, as we would like to examine the performance of the models in the unresolved setting.  We compare the results with Direct Numerical Simulation (DNS) \cite{Brachet}.  Note the total dissipation rate time history is quite inaccurate for each of the classical LES subgrid models.  In fact, employing no model leads to better performance than either the static Smagorinsky model, the dynamic Smagorinsky model, or the Chollet model.  The poor performance of the static Smagorinsky model is expected since it is active even at the beginning of the simulation when the solution is purely laminar.  The poor performance of the dynamic Smagorinsky model is more surprising as it was designed to improve upon the static Smagorinsky model when the flow is transitional.  The Chollet model dissipation rate time history matches the DNS closely until a time of approprixmately $t = 4$, at which point the two time histories quickly diverge.  This indicates the model is appropriately ``turned off'' when the mesh is sufficiently fine to resolve the flow field.

\begin{figure}[tp!]\centering
\includegraphics[width=0.48\textwidth]{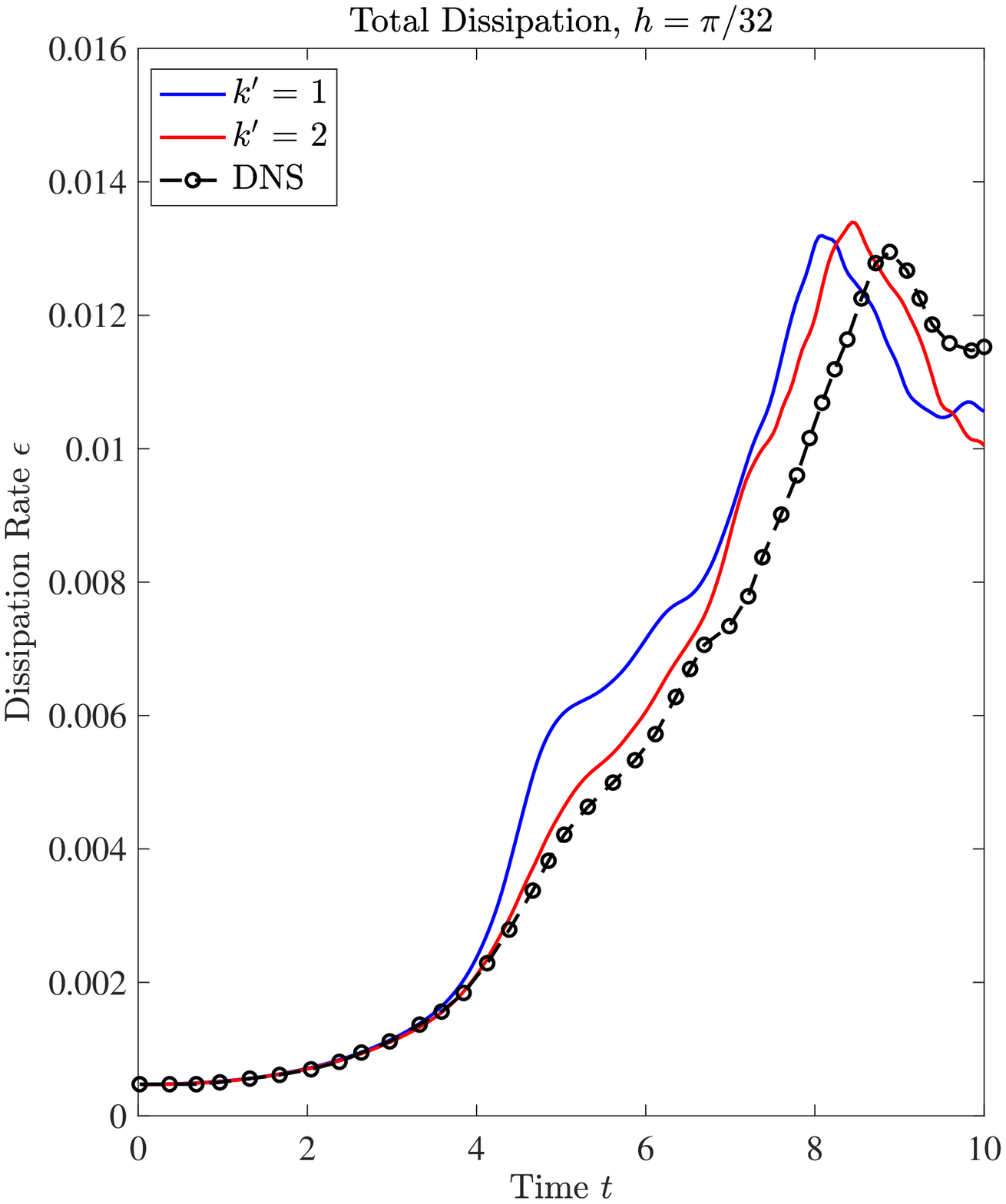}
\includegraphics[width=0.48\textwidth]{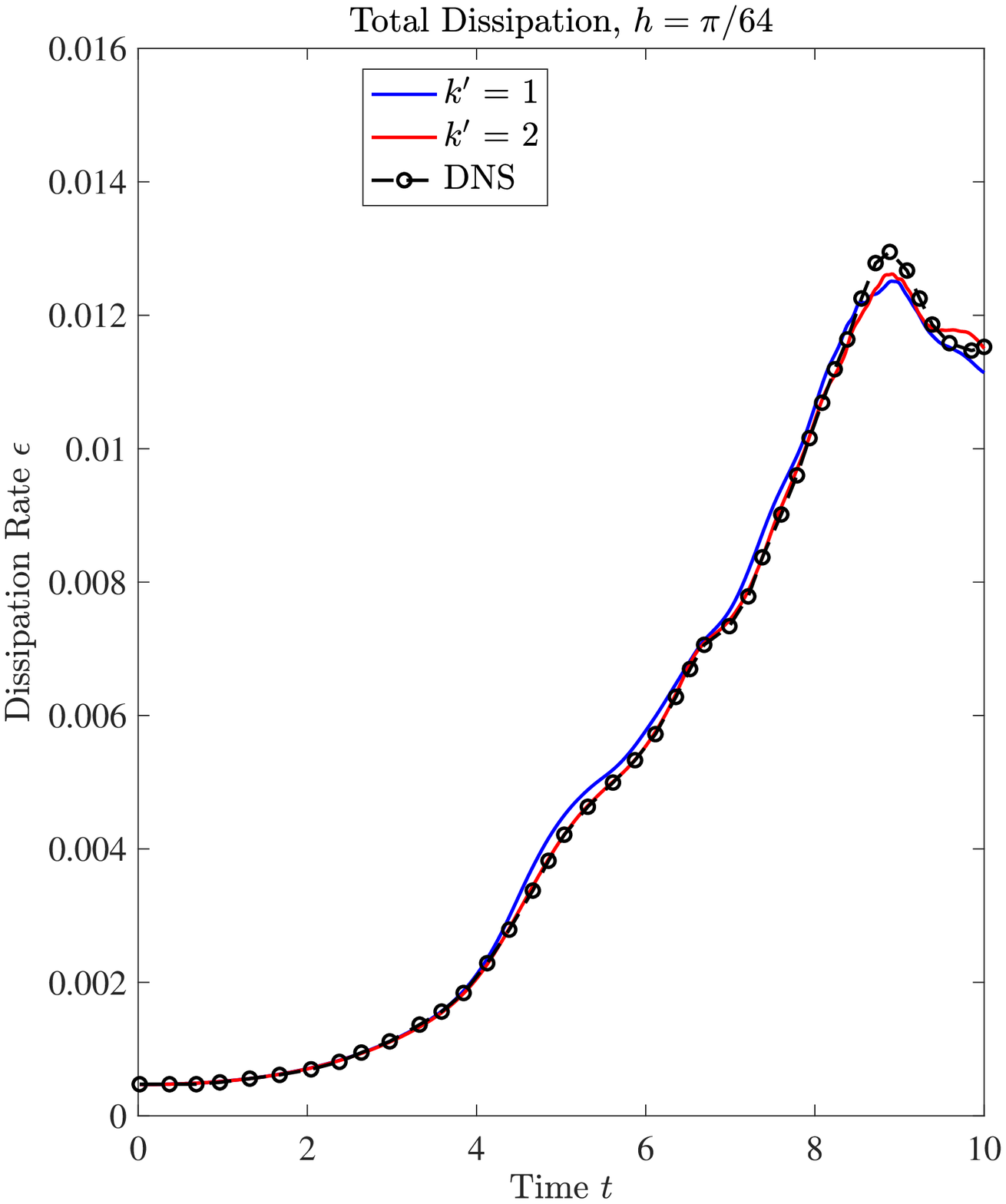}
\caption{Total dissipation rate time history for three-dimensional Taylor-Green vortex flow at $Re$ = 1600 using divergence-conforming B-splines of degree $k' = 1$ and $k' = 2$, $h = \frac{\pi}{32}$ (left) and $h = \frac{\pi}{64}$ (right), and our VMS-based formulation.}
\label{fig:3d-tgv-total}
\end{figure}

We next examine results obtained using our VMS-based formulation.  In Fig. \ref{fig:3d-tgv-total}(left), the total dissipation rate time history is displayed for our VMS-based formulation, divergence-conforming discretizations of degrees $k' = 1$ and $k' = 2$, and a mesh of $32^3$ elements.  From the figure, it is clear that the VMS-based formulation outperforms the static Smagorinsky model, the dynamic Smagorinsky model, and the Chollet model, using both $k ' = 1$ and $k' = 2$.  The VMS-based formulation also outperforms the case of no model.  The VMS-based formulation departs from the DNS at roughly the same point in time for both $k' = 1$ and $k' = 2$, indicating that the implicit subgrid model is appropriately ``turned off'' when the mesh is sufficiently fine to resolve the flow field.  Like the Chollet model, the VMS-formulation model dissipation rate time history begins to diverge from the DNS time history at a time of approximately $t = 4$, though the $k' = 2$ total dissipation rate time history matches the DNS much more closer than the $k' = 1$ total dissipation rate time history to the time of max dissipation rate.

In order to better understand how the VMS-based formulation responds to turbulence production and transition, we examine the resolved and model dissipation rates.  In Fig. \ref{fig:3d-tgv-resolved}(left) and Fig. \ref{fig:3d-tgv-model}(left), the resolved and model dissipation rate time histories are displayed for our VMS-based formulation, divergence-conforming discretizations of degrees $k' = 1$ and $k' = 2$, and a mesh of $32^3$ elements.  From these figures, it is clear that indeed the implicit subgrid model is ``turned off'' until a time of approximately $t = 4$, as the model dissipation rate is nearly zero until this time.  The model dissipation rate then sharply increases for $k' = 1$, and it increases much more gradually for $k' = 2$.  This is consistent with the typical observation that high-order methods exhibit less numerical diffusion than low-order methods.

\begin{figure}[tp!]\centering
\includegraphics[width=0.48\textwidth]{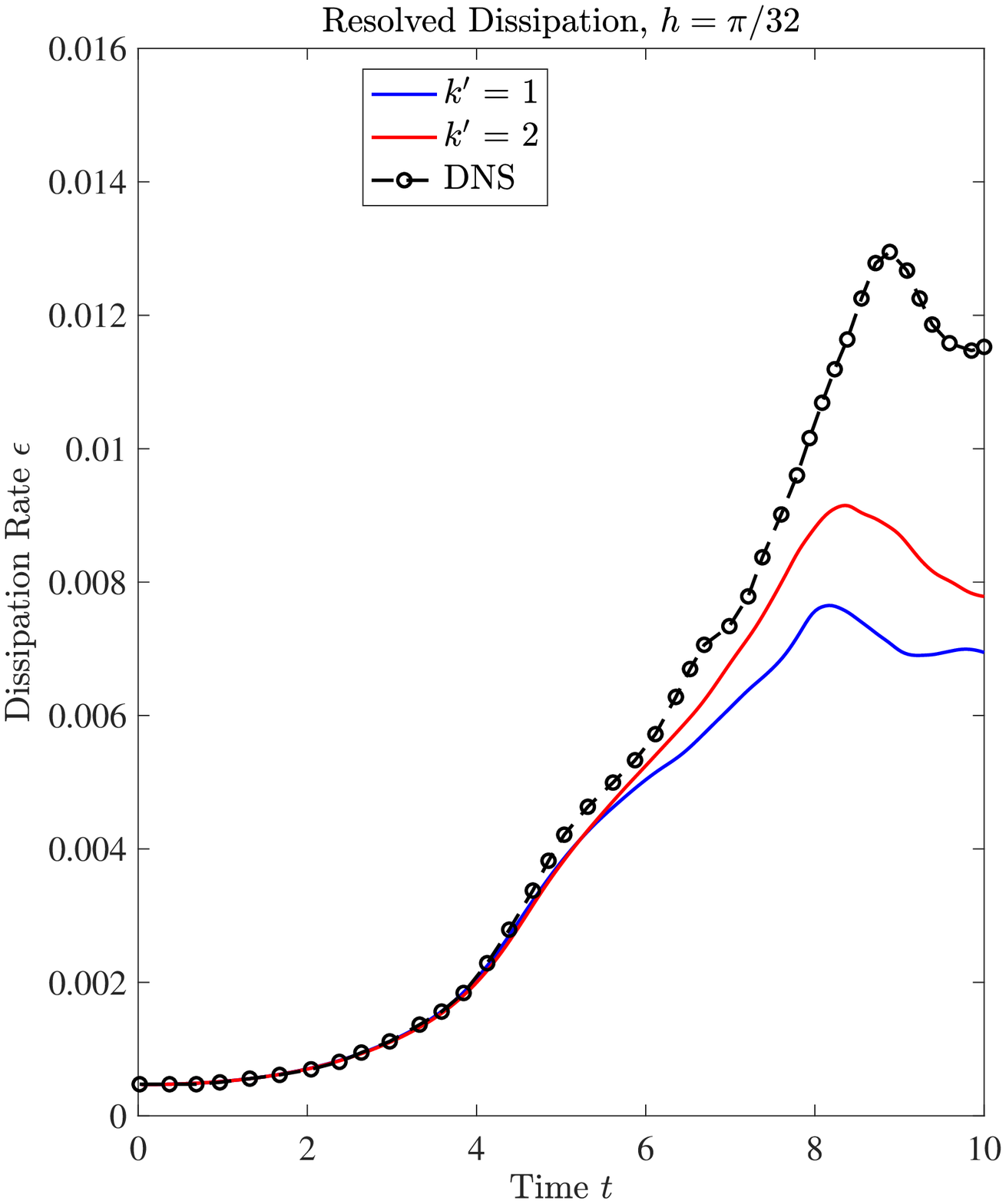}
\includegraphics[width=0.48\textwidth]{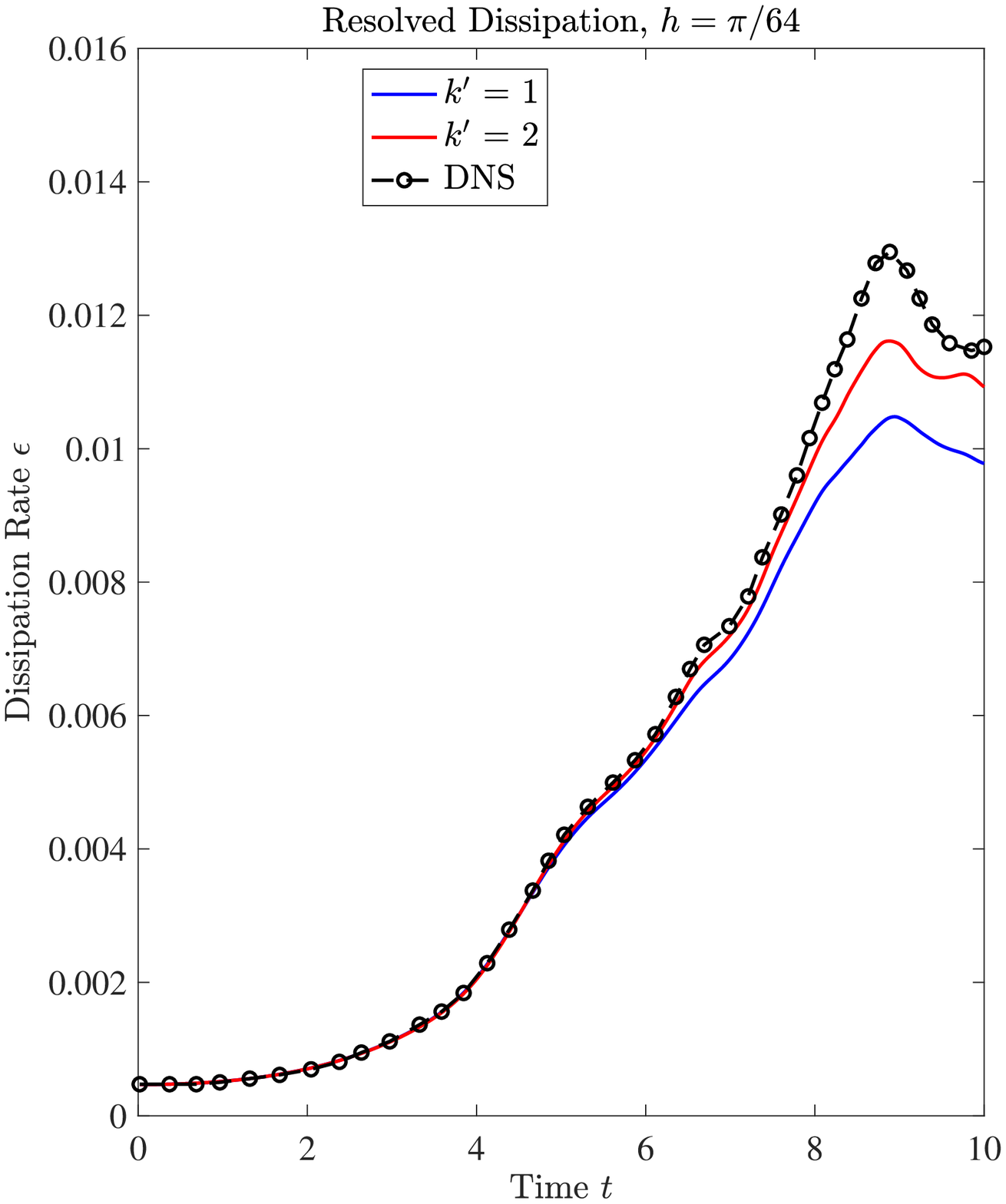}
\caption{Resolved dissipation rate time history for three-dimensional Taylor-Green vortex flow at $Re$ = 1600 using divergence-conforming B-splines of degree $k' = 1$ and $k' = 2$, $h = \frac{\pi}{32}$ (left) and $h = \frac{\pi}{64}$ (right), and our VMS-based formulation.}
\label{fig:3d-tgv-resolved}
\end{figure}

\begin{figure}[tp!]\centering
\includegraphics[width=0.48\textwidth]{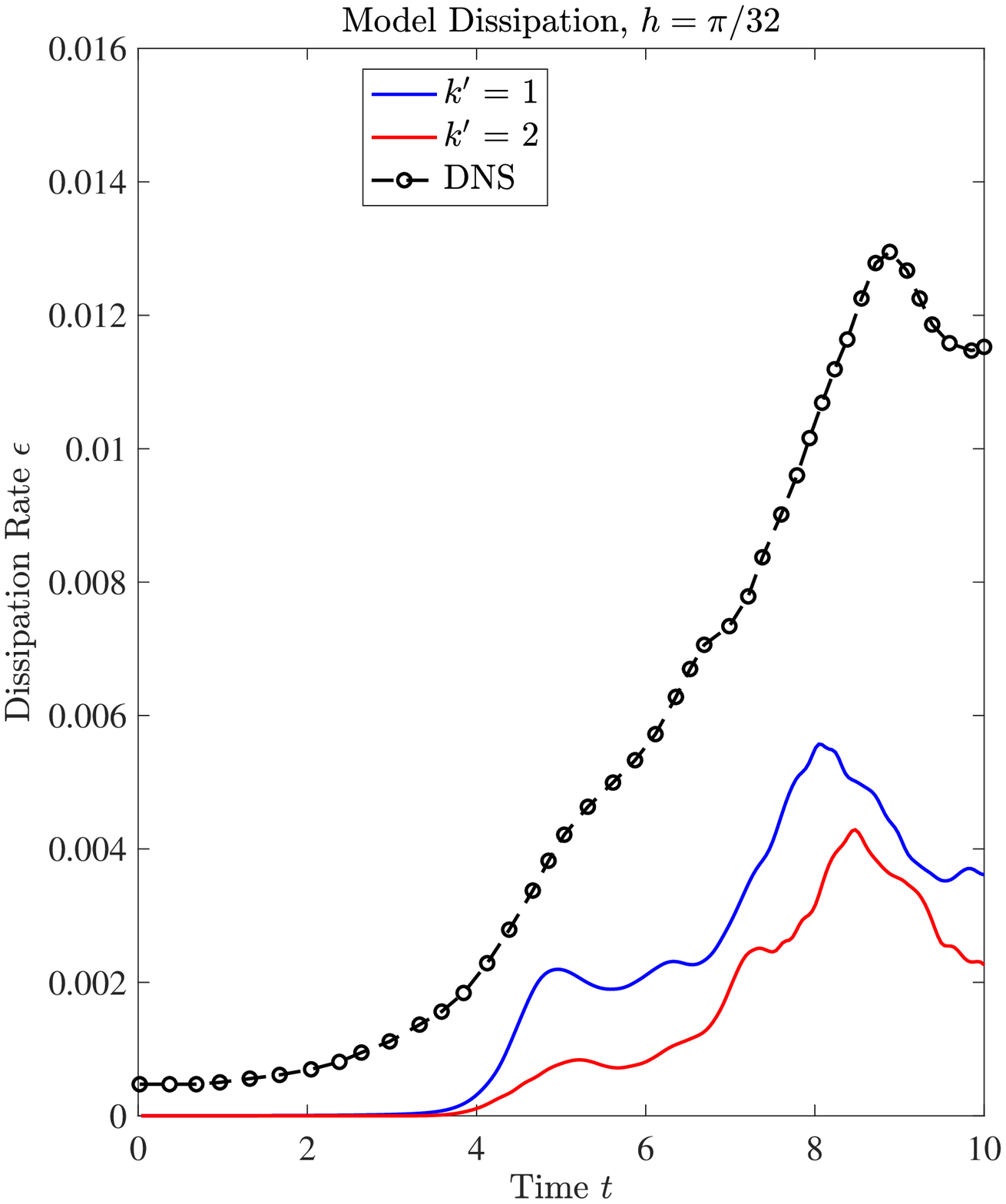}
\includegraphics[width=0.48\textwidth]{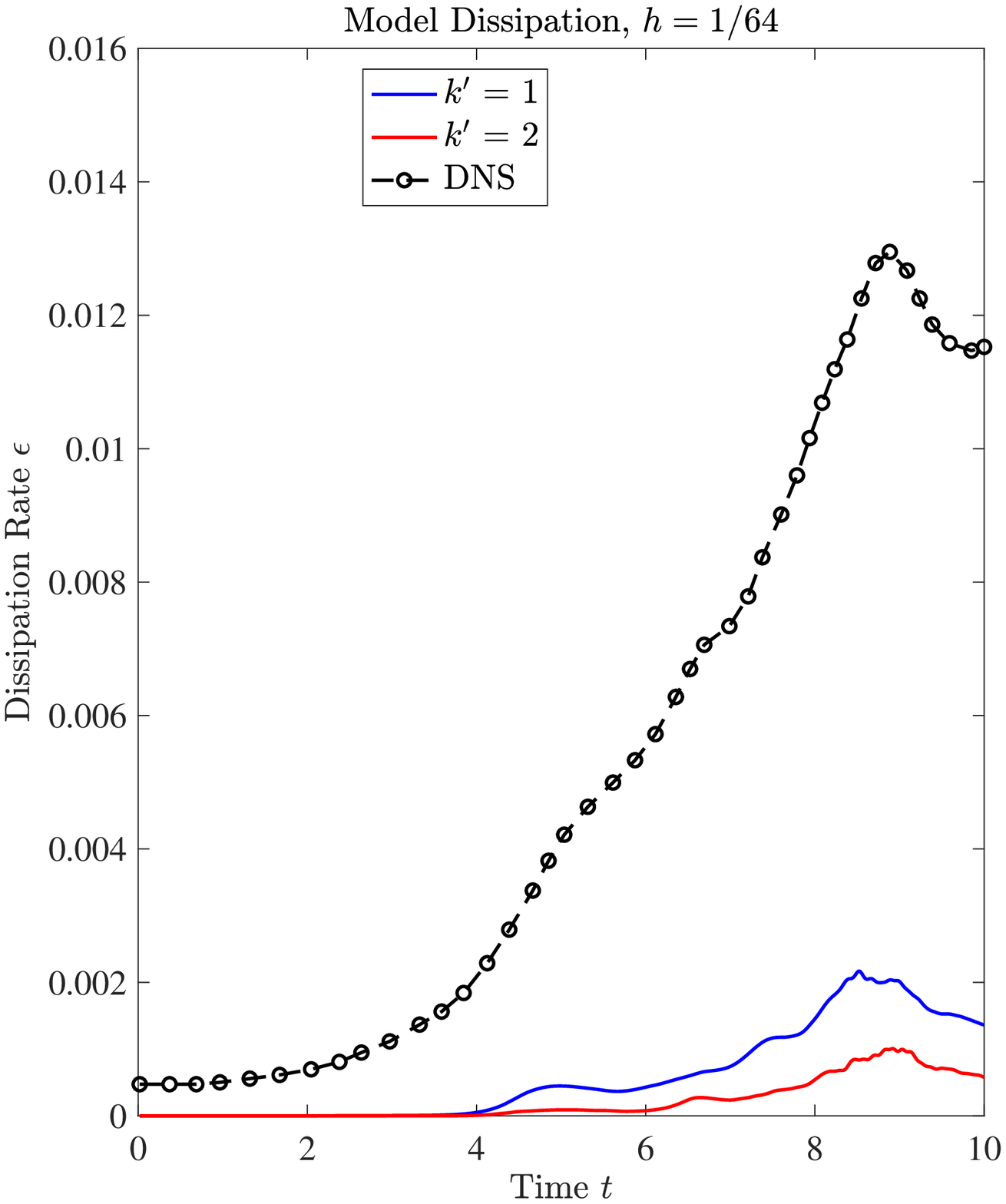}
\caption{Model dissipation rate time history for three-dimensional Taylor-Green vortex flow at $Re$ = 1600 using divergence-conforming B-splines of degree $k' = 1$ and $k' = 2$, $h = \frac{\pi}{32}$ (left) and $h = \frac{\pi}{64}$ (right), and our VMS-based formulation.}
\label{fig:3d-tgv-model}
\end{figure}

We next examine the impact of mesh refinement on our VMS-based formulation.  In Fig. \ref{fig:3d-tgv-total}(right), Fig. \ref{fig:3d-tgv-resolved}(right), and Fig. \ref{fig:3d-tgv-model}(right), the total, resolved, and model dissipation rate time histories are displayed for our VMS-based formulation, divergence-conforming discretizations of degrees $k' = 1$ and $k' = 2$, and a mesh of $64^3$ elements.  For this mesh, the flow is marginally unresolved.  It is clear here that as the mesh is refined, the total dissipation rate for both $k' = 1$ and $k' = 2$ converge to the DNS solution.  This is accompanied by an appropriate increase in the resolved dissipation rate as well as a corresponding decrease in the model dissipation rate.

\begin{figure}[t!]\centering
\includegraphics[width=0.48\textwidth]{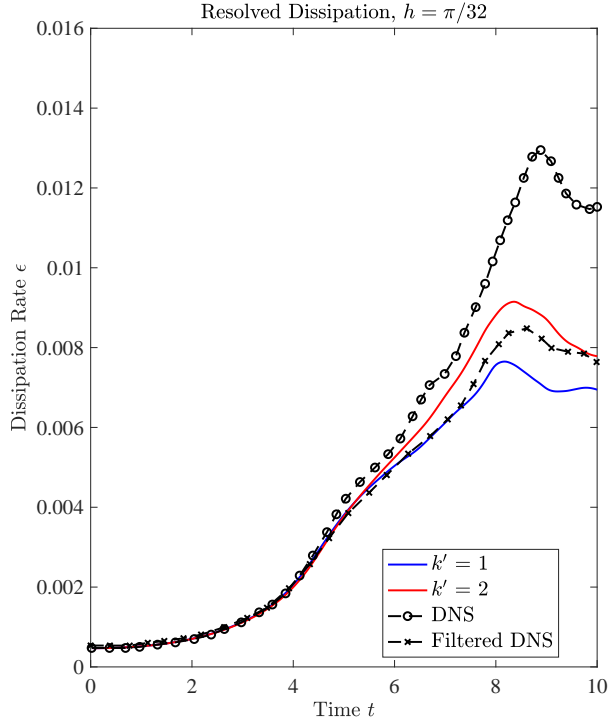}
\caption{Time history of resolved dissipation rate for simulation of the Taylor-Green vortex at $Re$ = 1600 with $k' = 1$ and $k' = 2$ and $h = \frac{\pi}{32}$.  Comparison of our VMS-based formulation with filtered DNS data \cite{Fauconnier08}.}
\label{fig:3d-tgv-comparison}
\end{figure}

Up to this point, we have compared simulation results directly with DNS data.  However, for unresolved simulations, such as the ones reported here for a mesh of $32^3$ elements, the energy in the coarse-scale velocity field (as defined by the projection of the exact velocity field into the coarse-scale space) is expected to be significantly less than the energy in the total velocity field.  As such, rather than compare the time history of the total dissipation rate for a given model with that of a DNS, it is likely more appropriate to compare the time history of the resolved dissipation rate with that associated with filtered or projected DNS data.  With this in mind, the time histories of the resolved dissipation rate for our VMS-based formulation for $k' = 1$ and $k' = 2$ and a mesh size of $h = \frac{\pi}{32}$ are displayed in Fig. \ref{fig:3d-tgv-comparison} alongside the time history of the dissipation rate computed from filtered DNS data \cite{Fauconnier08}.  The filtered DNS data was obtained using a sharp cut-off filter of size $\Delta = \frac{\pi}{32}$.  From the figure, we see that the resolved dissipation rate time histories associated with our VMS-based formulation closely match the resolved dissipation rate time history computed from the filtered DNS data.  This is evidence that our VMS-based formulation might be providing the correct amount of dissipation to yield the correct coarse-scale velocity field.  However, it should be noted that the filtered DNS data does not correspond to the projection of the exact velocity field into the coarse-scale space for either $k' = 1$ or $k' = 2$ but rather the projection of the exact velocity field into a space of Fourier modes, so no precise conclusions should be drawn.

%We now explore the applicability of the proposed stabilization as an implicit LES filter, using the well-known 3D Taylor--Green vortex problem (as detailed in \cite[Section 9]{Evans2018}).  For this, we employ the quasi-static subgrid model variant, discretized in space with div-conforming B-splines of degrees $k' = 1$ and $k'=2$, and in time with the implicit generalized-$\alpha$ method \cite{Chung1993}, as adapted to first-order systems by Jansen et al. \cite{Jansen2000}.  As in \cite{Evans2018}, we take the case of Re$=1600$ as a point of comparison, consider dissipation rate as a quantity of interest, and under-resolve in space to clearly show the effect of the model, using only $32^3$ B\'{e}zier elements.  We see from Figure \ref{} that the proposed stabilization provides a significant improvement over Galerkin's method.  We now consider a higher-resolution of $64^3$ B\'{e}zier elements, to illustrate the smooth transition to DNS provided by strongly-consistent stabilization.  In Figure \ref{}, we see that both the isogeometric and finite element computations are converging toward the DNS result with refinement.  

\section{Conclusions}\label{sec:conclusions}

In this paper, we have introduced a novel residual-based stabilized formulation for inf-sup stable discretizations of incompressible Navier-Stokes flow with $H^1$-conforming pressure approximation.  The method is formally derived using the VMS formalism using a Stokes projector for scale separation.  This yields a nonlinear algebraic system in which a discrete fine-scale pressure field must be solved for alongside the coarse-scale unknowns, such that the coarse- and fine-scale velocities are both discretely divergence-free.  The new formulation is energetically stable provided a dynamic subscale model is employed, and the formulation is quasi-optimal with respect to an SUPG/PSPG-like norm for divergence-conforming discretizations of a linearized model problem (Oseen flow).  Numerical results both support the analysis and provide initial evidence that the new formulation is capable of modeling the influence of unresolved flow features.  In the future, we plan to conduct a more thorough investigation of the turbulence modeling capabilities of the formulation and to study its interaction with weakly-enforced no-slip boundary conditions, on both conforming and immersed boundaries.  We also plan to extend the stability and convergence results presented here to non-divergence-conforming discretizations.

%{\color{red} Stolen from deleted section:} 
A limitation of the proposed stabilization is that it interferes with another desirable property of the Galerkin method using divergence-conforming spaces: pressure-robustness \cite{John2017}.  Briefly, pressure-robustness means that the error in the discrete velocity solution does not depend on the interpolation error of the pressure.  Lack of pressure robustness often manifests as poor mass conservation around large pressure gradients induced by irrotational source terms, such as those corresponding to immersed boundaries \cite{Kamensky2017a,Casquero2018}.  However, if the main practical advantage of pressure robustness is its prevention of un-physical mass loss, is it still an important property for divergence-conforming discretizations?  Does pressure interpolation error pollute velocity solutions in other important ways?  We plan to study these questions in future work, by applying the method of this paper in conjunction with immersed boundaries.

Finally, with the express goal of obtaining a residual-based stabilized formulation that maintains both strong mass conservation and pressure robustness for divergence-conforming spaces, we plan to apply the formalism presented here to the construction of VMS-based formulations yielding pointwise divergence-free coarse- and fine-scale velocities.  While such formulations have been proposed in the past \cite{VanOpstal2017,evans2018residual}, an error analysis which shows robustness in the advection-dominated limit is missing for these formulations.  The formalism here, however, may yield pressure-robust methodologies that are provably robust with respect to both advection and diffusion.

\section*{Acknowledgements}
%
%[DK] TODO: Funding stuff.
%
We thank the Texas Advanced Computing Center (TACC) at the University of Texas at Austin for providing HPC resources that contributed to this research.  FEniCS and tIGAr were run on TACC resources using Singularity \cite{Kurtzer2017} images converted from Docker \cite{Merkel2014} containers \cite{Hale2017}.

\bibliographystyle{unsrt}
\bibliographystyle{plain}
\bibliography{main}

\end{document}